\documentclass[10pt]{article}
\title{Regular Subgradients of Marginal Functions with Applications to Calculus 
and Bilevel Programming} 
\author{Le Phuoc Hai\thanks{Instituto de Alta Investigaci\'on (IAI), 
Universidad de Tarapac\'a, Arica, Chile. E-mail: lephuochai88@gmail.com.}
\and
Felipe Lara\thanks{Instituto de Alta Investigaci\'on (IAI), Universidad de 
Tarapac\'a, Arica, Chile. E-mail: felipelaraobreque@gmail.com, 
flarao@academicos.uta.cl; web: felipelara.cl. Research of this author was partially
supported by ANID--Chile under project Fondecyt Regular 1241040.}
\and
Boris S. Mordukhovich\thanks{Department of Mathematics, Wayne State 
University, Detroit, Michigan, USA. E-mail: aa1086@wayne.edu.  Research of 
this author was partly supported by the  US National Science Foundation under 
grant DMS-2204519, by the Australian Research Council under Discovery 
Project DP190100555, and by the Project 111 of China under grant D21024.}}
\date{\small \emph{\today}}
\usepackage{amsmath}
\usepackage{amsthm}
\usepackage{amssymb}
\usepackage{multicol}
\usepackage[active]{srcltx}
\usepackage{algorithm}
\usepackage{array}
\usepackage[dvipsnames,svgnames]{xcolor}
\usepackage[notcite,notref]{showkeys}

\numberwithin{equation}{section}

\newcommand{\R}{\mathbb{R}}

\def\ph{\varphi}

\def\inte{\mathop{\rm int}\nolimits}

\def\Lsup{\mathop{\rm Limsup}\nolimits}
\def\ep{\mathop{\rm epi}\nolimits}
\def\gr{\mathop{\rm gph}\nolimits}

\def\ri{\mathop{\rm ri}\nolimits}

\newtheorem{theorem}{Theorem}[section]
\newtheorem{corollary}[theorem]{Corollary}

\newtheorem{lemma}[theorem]{Lemma}
\newtheorem{example}[theorem]{Example}
\newtheorem{definition}[theorem]{Definition}
\newtheorem{remark}[theorem]{Remark}

\DeclareMathOperator{\dom}{dom}
\DeclareMathOperator{\epi}{epi}

\begin{document}
\maketitle
\begin{abstract}

\noindent {\bf Abstract}. The paper addresses the study and applications of a broad class of extended-real-valued functions, known as optimal value or marginal functions,  which frequently appear in variational analysis, parametric optimization, and a variety of applications. Functions of this type are intrinsically nonsmooth and require the usage of tools of generalized differentiation. The main results of this paper provide novel evaluations and exact calculations of regular/Fr\'echet subgradients and their singular counterparts for general classes of marginal functions via their given data. The obtained results are applied to establishing new calculus rules for such subgradients and necessary optimality conditions in bilevel programming.\\[1ex]
\noindent{\small {\bf Keywords}: Variational analysis, generalized differentiation, marginal functions, calmness, variational calculus, bilevel programming\\[1ex]
}
\end{abstract}

\medskip

\centerline{Dedicated to Fabi\'an Flores-Baz\'an on the occasion of his 60th Birthday.}

\section{Introduction}\label{intro}
 
Our main interest in this paper is attracted to the class of {\em optimal 
value/marginal functions} given in the form
\begin{equation}\label{1722023_1}
\mu (x):= \inf\big\{\varphi(x,y)\;\big|\; y \in G(x)\big\},
\end{equation}
where $\varphi:X \times Y \rightarrow{\mathbb{\bar R}}$ is a cost/objective 
function defined on the product of normed spaces $X$ and $Y$ and taking values 
in the extended real line $\bar{\mathbb{R}} := (-\infty,\infty]$, and where 
$G : X \rightrightarrows Y$ is a constraint set-valued mapping. The {\em
 solution/argminimum mapping} $M: X \rightrightarrows Y$  associated with
 \eqref{1722023_1} is defined by
\begin{equation}\label{1722023_2}
M(x):=\big\{y \in G(x)\;\big|\;\;\varphi(x, y) = \mu(x)\big\}.
\end{equation}

Over the years, it has been well recognized in variational analysis and optimization 
that the class of marginal functions \eqref{1722023_1} plays a highly important role 
in many aspects of variational theory and applications; see, e.g.,  the books
\cite{DontchevRockafellar,Mor2006,M2018,M24,MN2022,Book_Penot,RW,Thibault} 
with the discussions and references therein. It follows from the definition that
\eqref{1722023_1} can be viewed as the optimal value in the {\em parameterized
problem} of minimizing the cost function $\ph(x,y)$ with respect to the decision 
variable $y$ subject to the parameter-dependent constraint set $G(x)$. Besides 
this obvious interpretation, marginal functions arise in various settings that may 
not be related to optimization; see the books mentioned above. 

A characteristic feature of \eqref{1722023_1} is the {\em intrinsic nondifferentiability} of this function even in the case of smooth objectives and simple convex constraint sets. Therefore, sensitivity analysis of marginal functions requires the usage of suitable machinery of {\em generalized differentiation}, which has been done at some levels for various generalized derivatives under diverse assumptions; see, e.g.,  \cite{GO,KL-2,Mor2006,M2018,MN2022,LK,Book_Penot,Thibault} and the bibliographies therein for more details.

In this paper, we focus on the study and applications of {\em regular/Fr\'echet subgradients} and {\em singular regular subgradients} of marginal functions given in the general form \eqref{1722023_1}. These subdifferential constructions for marginal functions seem to be much less investigated in comparison with the classical subdifferential of convex analysis, the generalized gradient of Clarke, and the limiting and singular subdifferentials of Mordukhovich. Some results for regular subgradients of marginal functions can be found in \cite{HuxuAn,MNY2009} and are discussed in what follows. The results of our paper significantly improve the known ones while providing also completely new conditions for lower and upper estimates and exact calculation formulas for both regular and singular regular subdifferentials of marginal functions in general settings of normed spaces. Moreover, we present applications of the results obtained for marginal functions to deriving new {\em calculus rules} for regular subgradients and their singular counterparts with further developing the {\em marginal function approach} to optimistic {\em bilevel programming}. \vspace*{0.03in}

The rest of the paper is organized as follows. Section~\ref{sec:2} contains some preliminaries from variational analysis and generalized differentiation together with definitions of the major constructions and notions that are largely used to establish our main results in the subsequent sections. 

In Section~\ref{sec:3}, we derive new results on {\em lower and upper estimates} for both regular and singular regular subdifferentials of the general class of marginal functions \eqref{1722023_1} on normed spaces that are expressed via the given data of \eqref{1722023_1}. To the best of our knowledge, the lower subdifferential estimates for marginal functions have never been obtained in the literature outside of equalities. Our results in this direction are established under the alternative assumptions of either the newly defined notion of {\em nearly-isolated calmness} for constraint mappings, or {\em upper Lipschitzian selections} of the solution mapping \eqref{1722023_2}.

Section~\ref{sec:4} establishes {\em exact calculation formulas} for the regular and singular regular subdifferentials of marginal functions on arbitrary normed spaces. Besides the aforementioned assumptions of Section~\ref{sec:3}, the obtained exact subdifferential calculations require either the Fr\'echet differentiability of the cost function, or certain {\em metric qualification conditions} imposed on the initial data of \eqref{1722023_1}. All these assumptions are illustrated by examples. It has been well realized in variational analysis that developing {\em variational calculus rules} is crucial for applications of any subdifferential construction. Not much has been done in this direction for regular subgradients and their singular counterparts. 

Section~\ref{sec:5} is devoted to applications of the subdifferential formulas for marginal functions obtained in Section~\ref{sec:4} to deriving {\em exact/equality type} chain, sum, product, and quotient rules for regular and singular regular subgradients of functions on normed spaces. Some of the calculus results of this type for regular subgradients were obtained in \cite{MNY2006} under the local Lipschitz continuity of the functions involved. Besides establishing novel calculus rules, we systematically replace the Lipschitz continuity assumptions in the results of \cite{MNY2006} by the corresponding {\em calmness} properties, which lead us to refined calculus.  

Section~\ref{sec:bilevel} develops the {\em marginal function approach} to deriving {\em optimality conditions} in problems of {\em optimistic bilevel programming} in {\em Asplund spaces}, which constitutes a broad class of Banach spaces that includes, in particular, all reflexive ones. Involving subdifferentiation of marginal functions and calculus rules for the regular and limiting subdifferentials together with the aforementioned metric qualification conditions, we establish new necessary optimality conditions for local optimal solutions to optimistic bilevel programs in generally non-Lipschitzian settings, which significantly improve those obtained in \cite{Dempe2020} for problems with fully Lipschitzian data in finite-dimensional spaces under additional conditions and more restrictive constraint qualifications.

In Section~\ref{sec:7}, we present concluding remarks and discuss some directions of our future research and applications.\vspace*{-0.15in}

\section{Preliminaries and Basic Definitions}\label{sec:2}

Throughout the paper, we use standard notation of variational analysis and generalized differentiation; see, e.g., \cite{Mor2006}. Given a {\em normed space} $X$ endowed with the norm $\|\cdot\|$, the topologically dual space of $X$ is denoted by $X^*$, where $w^*$ stands for the weak$^*$ topology of $X^*$ and $\langle \cdot,\cdot \rangle$ indicates the canonical pairing between $X$ and $X^*$. The symbols $\mathbb{B} (\bar x, r)$ and 
 ${\mathbb{\overline B}}_{X^*}$ signify the open ball centered at 
 $\bar{ x}$ with radius $r > 0$ and the closed unit ball of $X^*$, respectively. 
 We also use the notations $\mathbb{R}_{+} := [0,\infty)$, 
 $\mathbb{R}_{-} := (-\infty,0]$ and $x \xrightarrow{S} \bar{x}$ when 
 $x \to \bar x$ with $x \in S$.

 For a set-valued mapping $F : X \rightrightarrows X^{*}$, the symbol
\begin{align*}
\Lsup\limits_{x \rightarrow \bar{ x}} \, F(x) := \big\{ x^{*} \in X^{*}\;\big|\;
\exists \, x_k \to \bar x,\,\,\, x^*_k \xrightarrow{w^*} x^* \,\,\, 
\text{with} \,\,\,  x^{*}_k \in F(x_k),\; k \in \mathbb{ N}\big\}
\end{align*}
stands for the (Painlev\'e-Kuratowski) {\em sequential outer limit} of $F$ as 
$x \to\bar x$ with respect to the norm topology of $X$ and the weak$^*$ topology of the dual space $X^{*}$, where $\mathbb{ N}:=\{1,2,\ldots\}$ signifies the collection of natural numbers.
 
Given a set-valued mapping/multifunction $F : X \rightrightarrows Y$ between normed spaces, the {\em domain} and {\em graph} of $F$ are defined by
\begin{equation*}
\dom F :=\big\{x \in X\;\big|\;F(x) \ne \emptyset\big\},\;
\gr F :=\big\{(x, y) \in X \times Y\;\big|\;\;y \in F(x)\big\},
\end{equation*}
respectively. Unless otherwise stated, the {\em norm} on the product space $X\times Y$ is defined by $\|(x, y)\|: = \|x\| + \|y\|$ for any $(x, y) \in X \times Y$. 

Given an extended-real-valued function $f : X \to\mathbb{\bar R}$ on a normed space $X$, its {\em domain} and {\em epigraph} are defined, respectively, by
\begin{equation*}
\dom f:=\big\{x\in X\;\big|\;f(x)\ne\emptyset\}\;\mbox{ and }\;\ep f:=\big\{(x,\alpha)\in X\times\mathbb R\;\big|\;\alpha\ge f(x)\big\}.
\end{equation*}

Following the terminology of \cite{RW}, the (Fr\'echet) {\em regular subdifferential}  $\Hat\partial f(\bar x)$ of $f$ at $\bar x\in\dom f$ is the collection of subgradients $x^*\in X^*$ such that for every 
$\varepsilon > 0$ there exists $\gamma> 0$ satisfying
\begin{align}\label{eq:Fdef}
f(x) - f(\bar{x}) - \langle x^{*}, x - \bar{x} \rangle \geq - \varepsilon 
\|x - \bar{x}\|\;\mbox{ whenever }\;x \in \mathbb{B} (\bar{x},\gamma).
\end{align}
Along with $\widehat\partial f(\bar x)$, we consider also the {\em upper regular subdifferential} of $f\colon X\to\mathbb R\cup\{-\infty\}$ at $\bar x$ with $|f(\bar x)|<\infty$ defined symmetrically as
\begin{equation}\label{upper-sub}
\widehat\partial^+f(\bar x):=-\widehat\partial(-f)(\bar x).
\end{equation}
As well known \cite[Proposition~1.87]{Mor2006}, the sets $\widehat\partial f(\bar x)$ and $\widehat\partial^+f(\bar x)$ are nonempty simultaneously if and only if $f$ is Fr\'echet differentiable at $\bar x$.

The corresponding {\em regular normal cone} to a set $S\subset X$ at $\bar x\in S$ is given by
\begin{align}\label{Fnor}
\widehat N (\bar{x};S) := \Big\{x^* \in X^*\;\Big|\;\limsup\limits_{x 
\xrightarrow{S} \bar{x}} \Big\langle x^*, \dfrac{x - \bar{x}}{\|x - 
\bar{x}\|}\Big\rangle\le 0 \Big\}
\end{align}
with $\widehat N(\bar x;S):=\emptyset$ if $\bar x\notin S$. Note that the regular subdifferential \eqref{eq:Fdef} of $f$ at $\bar x\in S$ can be geometrically expressed via the regular normal cone \eqref{Fnor} by
\begin{align*}
\widehat {\partial} f(\bar{x}) =\big\{ x^* \in X^*\;\big|\;(x^*, - 1) \in 
\widehat {N}\big((\bar{ x}, f(\bar{x}));\ep f\big)\big\}.
\end{align*}
Along with the regular subdifferential \eqref{eq:Fdef}, we consider the {\em singular regular subdifferential} of $f$ at $\bar x\in\dom f$ defined by
\begin{align}\label{eq:Fsin}
\widehat{\partial}^\infty f(\bar{x}) :=\big\{ x^* \in X^*\;\big|\;(x^*,0) \in 
\widehat {N}\big((\bar{ x}, f(\bar{x}));\ep f\big)\big\},
\end{align}
which can be equivalently described as follows: $x^*\in \widehat {\partial}^{\infty}f(\bar{ x})$ if and only if  for any $\varepsilon>0$ there exists $\gamma>0$ such that
\begin{align*}
\langle x^{*}, x - \bar{x} \rangle \le \varepsilon\big(\|x - \bar x\| + 
\lvert \beta - f(x) \rvert \big)\;\mbox{ whenever }\;\big(x, \beta) \in \ep f \cap \mathbb B 
 ((\bar x, f (\bar{x})),\gamma\big).
\end{align*} 
For a set $S\subset X $ with its {\em indicator function} $\delta_S\colon X\to\mathbb{R}$ defined by $\delta_S(x):=0$ if $x\in S$ and $\delta_S(x):=\infty$ if $x\notin S$, we have the relationships
\begin{align*}
\widehat N (\bar{ x};S) = \widehat{\partial}\delta_S (\bar{x}) = \widehat {\partial}^{\infty} 
\delta_S (\bar{x})\;\mbox{ for all }\;\bar x\in S.
\end{align*}

Now we define, following the scheme from \cite{Mor1980}, a counterpart of the regular constructions \eqref{eq:Fdef} and \eqref{Fnor} for set-valued mappings between normed spaces. The {\em regular coderivative} of $F : X \rightrightarrows Y$ at $(\bar x,\bar y)\in X\times Y$ is the multifunction $\widehat {D}^{*} F (\bar{x}, \bar{y})\colon Y^{*} 
\rightrightarrows X^{*}$ with the values
\begin{align}\label{cod}
\widehat {D}^{*} F(\bar{x}, \bar{y}) (y^{*}) :=\big\{x^{*} \in X^{*}\;\big|\;
(x^{*}, - y^{*}) \in \widehat N\big((\bar{x}, \bar{y});\gr F\big)\}
\end{align}
if $(\bar x,\bar y)\in\gr F$, and $\widehat {D}^{*} F(\bar{x}, \bar{y}) (y^*):= \emptyset$ if 
$(\bar{x}, \bar{y}) \notin \gr F$. We omit  $\bar y$ in the coderivative notion if $F(\bar x)$ is a singleton. Observe that
\begin{align*}
\widehat {D}^{*} F(\bar{x}) (y^*) =\big\{\nabla F(\bar{ x})^{*}y^{*}\big\}\;\mbox{ for all }\;y^*\in Y^*
\end{align*}
if $F$ is Fr\'echet differentiable at $\bar{x}$ via the adjoint derivative operator $\nabla F(\bar{ x})^{*}$.\vspace*{0.05in}

Recall that a mapping $h:D\subset X \to Y$ is {\em upper Lipschitzian} \cite{Robinson} at $\bar x\in D$ or {\em calm} \cite{RW} at this point if there exist numbers $r>0$ and $\ell\ge 0$
such that 
\begin{align}\label{calm}
\|h(x) - h(\bar{ x})\| \le\ell\|x-\bar{ x}\|\;\mbox{ for all }\;x \in \mathbb{B} 
(\bar{x}, r)\cap D.
\end{align} 
This notion is definitely weaker than the the Lipschitz continuity of $h$ around $\bar x$ as is illustrated by the simple example of $h:\mathbb{R} 
 \to\mathbb{ R}$ given by $h(x) := x^{3/2} \sin(1/x)$ for all $x \in 
 \mathbb{R} \setminus \{0\}$ and $h(0) := 0$. This function is calm at $\bar x=0$ while not Lipschitz continuous around this point.
 
 It is not difficult to deduce from the definitions (see, e.g., the proof of \cite[Proposition~3.5]{MNY2006}) that the calmness of $h\colon X\to Y$ at $\bar x$ yields the {\em regular coderivative scalarization formula}
\begin{equation}\label{scal}
\widehat D^* h(\bar{ x})(y^*)=\widehat {\partial} \langle 
 y^*, h\rangle (\bar{ x}) \;\mbox{ for all }\;y^*\in Y^* 
\end{equation}
in any normed spaces $X$ and $Y$. 

Given further a set-valued mapping $F: X \rightrightarrows Y$ between normed spaces, we say as in \cite{DontchevRockafellar} that $F$ enjoys the {\em isolated calmness} property at $(\bar{x}, \bar{y})\in \gr F$  with modulus $L> 0$ if there exists $\gamma>0$ such that
\begin{equation}\label{i-calm}
F(u) \cap \mathbb{B}_{Y}(\bar{y}, \gamma ) \subset \{\bar{y}\} + L 
\|u-\bar{x}\|{\mathbb{\overline B}}_{Y}\;\mbox{ for all }\;u \in \mathbb{B}_{X} 
(\bar{x}, \gamma).
\end{equation}

Next we formulate and discuss certain {\em metric qualification conditions} the importance of which has been well realized in different aspects of variational analysis, generalized differential calculus, and optimization theory; see below.

\begin{definition}[\bf metric qualification conditions]\label{mqc} {\rm
 Let $C,D$ be subsets of a normed space $X$, and let $\bar{x} \in C \cap D$. 
It is said that these sets satisfy:
\begin{description}
\item[\bf(i)]  The {\em metric qualification condition \eqref{Q1}} at $\bar{x}$ if there exist $a, r > 0$ such that
\begin{align}\label{Q1}
d({x}, C \cap D) \le ad({ x},C) + ad({ x}, D)\;\mbox{ for all }\;x \in \mathbb{B} 
(\bar{x}, r), \tag{$Q_1$}
\end{align}
where $d(x,S)$ stands for the distance from $x$ to the set $S$.
		
\item[\bf(ii)] The {\em metric qualification condition \eqref{Q2}} at $\bar{x}$ with respect to some subdifferential $\partial^{\textasteriskcentered}$ if there exists $a > 0$ such that
\begin{align}\label{Q2}
\partial^{\textasteriskcentered} d(\bar{x}, C \cap D) \subset a \partial^{\textasteriskcentered} d(\bar { x},C) + a \partial^{\textasteriskcentered} 
d(\bar{x}, D). \tag{$Q_2$}
\end{align}
\end{description}}
\end{definition}

The terminology of Definition~\ref{mqc} is taken from \cite{HuxuAn} while 
such conditions are traced back to \cite{Ioffe} in the case of \eqref{Q1} and to 
\cite{jour} in the case of \eqref{Q2}. Observe that condition 
\eqref{Q2}, in contrast to \eqref{Q1}, depends on the choice of the 
subdifferential $\partial^{\textasteriskcentered}$. In \cite{HuxuAn,Ioffe,jour,NgaiThera,Book_Penot,Thibault} among others, the reader 
can find sufficient conditions for the fulfillment of \eqref{Q2} with respect to 
different subdifferentials and their applications. On the other hand, the 
subgradient-free condition \eqref{Q1}, which is easily to deal with, yields the 
fulfillment of \eqref{Q2} in some important subgradient settings; see, e.g., 
\cite{HuxuAn} and the references therein. Furthermore, note that the proper 
\eqref{Q1}  is also called subtransversality, linear regularity, or linear coherence 
in the literature. \vspace*{0.05in} 

We conclude this section with three lemmas broadly used in what follows. The 
first one, taken from \cite[Corollary~1.96]{Mor2006}, provides relationships 
between regular normals to sets and regular subgradients of the associated 
distance functions.

\begin{lemma}[\bf normals to sets and subgradients of distance  functions]\label{lem2722023_1} $\,$ Let $S \subset X$ be a nonempty subset 
of a normed space. Then for every $\bar x\in S$ we have the regular 
normal-subgradient relationships
\begin{align*}
\widehat{\partial} d(\cdot,S) (\bar{ x}) = \mathbb{B}_{X^{*}} 
\cap{\widehat N} (\bar{x}; S)\;\mbox{ and }\;
{\widehat N} (\bar{x}, S) = \bigcup_{\lambda>0}\lambda\widehat{\partial} d(\cdot, S) (\bar{x}). 
\end{align*}
\end{lemma}

The next result is a consequence of Lemma~\ref{lem2722023_1} for the epigraphical set $S:=\epi f$ with the replacement of  $\bar{ x}$ by $(\bar{ x}, f(\bar{ x}))$.

\begin{lemma}[\bf normals and subgradients for epigraphs]\label{12202023_lem2} $\,$Let $f\colon X\to\Bar{\mathbb R}$ be defined on a normed space, and let $\bar{x}\in S$. Then we have
\begin{align*}
\widehat{N}\big((\bar x,f(\bar x); \ep f)\big) = \bigcup_{\lambda>0} 
\lambda\widehat{\partial} d(\cdot, \ep f)\big(\bar x,f(\bar x)\big).
\end{align*}
\end{lemma}

The last lemma, taken from \cite[Proposition~1.88]{Mor2006}, provides a smooth variational description of regular subgradients for general functions on normed spaces.

\begin{lemma}[\bf smooth variational description of regular subgradients]\label{12202023_lem} $\,$
Let $f\colon X\to\Bar{\mathbb R}$ be defined on a normed space $X$, and let $\bar{x}\in\dom f$. Then the inclusion $x^*\in\widehat\partial f(\bar x)$ for regular subgradients is equivalent to the existence of a neighborhood $U$ of $\bar x$ and a function $s\colon U\to\mathbb R$ such that $s(\cdot)$ is Fr\'echet differentiable at $\bar x$ with $\nabla s(\bar x)=x^*$, that $s(\bar{ x}) = f (\bar x)$, and that $s(x) \le f(x)$ for all $x\in U$.
\end{lemma}\vspace*{-0.15in}

\section{Estimates for Regular and Singular Regular Subgradients of Marginal Functions}\label{sec:3}

In this section, we establish lower and upper estimates for regular subgradients and their singular counterparts of general marginal functions \eqref{1722023_1}. Without further mentioning, all the spaces under consideration are {\em arbitrarily normed}. 

The lower estimates for both regular and singular regular subgradients of marginal functions, new in the literature, are obtained below under the alternative conditions formulated in the following two definitions. \vspace*{0.05in}

The first notion, used in \cite{MNY2009} to derive equalities for regular (while not singular regular) subgradients of marginal functions, employs the upper Lipschitzian/calmness property of single-valued mappings defined in \eqref{calm}.

\begin{definition}[\bf upper Lipschitzian selections]\label{upper-lip} {\rm 
 A set-valued mapping $F : D \rightrightarrows Y$ defined on some set 
 $D \subset X$, admits an {\em upper Lipschitzian selection} at 
 $(\bar x, \bar y) \in \gr F$ if there exists a single-valued mapping $h: D \to  Y$, 
 which is calm at $\bar{ x}$ satisfying $h(\bar x) = \bar{y}$ and 
 $h(x) \in  F(x)$ for all $x \in D$ in a neighborhood of $\bar{ x}$.}
\end{definition}

The next definition introduces two new well-posedness properties of set-valued mappings that are related to isolated calmness in \eqref{i-calm}.
 
\begin{definition}[\bf nearly-isolated calmness]\label{df1102024_1}
{\rm Given a set-valued mapping $F: X \rightrightarrows Y$ with $(\bar{x}, \bar{y}) 
\in \gr F$, consider the following properties:
\begin{description}

\item[\bf(i)] $F$ is {\em locally nearly-isolated calm} at $(\bar{x}, \bar{y})$ with 
modulus $L>0$ if there exists $\gamma>0$ such that 
\begin{align}\label{eq:lnstrong}
F(u) \cap \mathbb{B}_{Y}(\bar{y},\gamma) \subset\{\bar{y}\} + L 
\|u - \bar{x}\| {\mathbb{\overline B}}_{Y}\;\mbox{ for all }\;u \in \mathbb{B}_{X} 
(\bar{x},\gamma)\setminus\{\bar x\}.
\end{align}

\item[\bf(ii)] $F$ is {\em semilocally nearly-isolated calm} at $(\bar{x}, \bar{y})$ with 
modulus $L>0$ if there exists $\gamma>0$ such that 
\begin{align}\label{eq:snstrong}
F(u) \subset \{\bar{y}\} + L \|u - \bar{x}\| {\mathbb{\overline B}}_{Y}\;\mbox{
for all }\;u \in \mathbb{B}_{X} (\bar{x}, \gamma) \setminus \{\bar x\}.
\end{align}
\end{description}}
\end{definition}

Clearly, the introduced properties are weaker than the isolated calmness \eqref{i-calm} even in simple examples as for $F\colon\mathbb R\rightrightarrows\mathbb{ R}$ 
defined by $F(x):=\R$ if $x=0$ and $F(x):=\emptyset$. This mapping satisfies \eqref{eq:lnstrong} but not \eqref{i-calm} at $(0,0)$. It is important to note that 
property \eqref{eq:snstrong} is a stronger condition than property \eqref{eq:lnstrong}.
For instance, consider  \( F \colon \mathbb{R} \rightrightarrows \mathbb{R} \) 
defined by \( F(x) = \{0\} \) if \( x = 0 \) and \( F(x) = \{1\} \). This mapping 
satisfies condition \eqref{eq:lnstrong}, but it does not satisfy condition
\eqref{eq:snstrong} at the point \( (0,0) \).

Besides,  the emergence of the notion of semilocal near-isolated calmness is quite 
natural because this notion is indeed connected to the notion of upper Lipschitz multifunctions, as mentioned in \cite{Zhangtreiman}. Semilocal near-isolated calmness represents a variant of upper Lipschitz continuity for a multifunction. Furthermore, the authors \cite{Zhangtreiman} have already provided criteria, similar to the Levy-Rockafellar criterion for isolated calmness (see \cite{DontchevRockafellar}), to verify this property using the so-called outer derivative of a multifunction (see \cite[Theorem 3.2]{Zhangtreiman})  and the contingent derivative of the multifunction (see \cite[Propositions 3.4, Corollaries 3.6 and 3.7]{Zhangtreiman}). 
\vspace*{0.05in}

Our first theorem presents both lower and upper estimates of {\em regular subgradients} \eqref{eq:Fdef} of marginal functions in the general normed space setting. 

\begin{theorem}[\bf estimates for regular subgradients]\label{thm1}
Let the marginal function $\mu(\cdot)$ in \eqref{1722023_1} be finite 
 at some $\bar{x} \in \dom M$ from \eqref{1722023_2}, and let $\bar{y} \in M(\bar{x})$. The
 following assertions are satisfied:  
 \begin{description}
\item[\bf(i)] If $\widehat{\partial}^{+} \varphi (\bar{x}, \bar{y})\ne \emptyset$ for the upper regular subdifferential \eqref{upper-sub}, then 
\begin{align}\label{1722023_3}
 \widehat{\partial} \mu(\bar{x}) \subset \bigcap_{(x^{\ast}, y^{\ast}) \in 
 \widehat{\partial}^{+} \varphi (\bar{x}, \bar{y})}\big[x^{\ast} + \widehat{D}^* 
 G(\bar{x}, \bar{y}) (y^{\ast})\big].
 \end{align}
	
\item[\bf(ii)] We have the lower subgradient estimate
 \begin{align}\label{1722023_4}
\widehat{\partial} \mu (\bar{x})\supset\bigcup_{(x^{\ast}, y^{\ast}) \in \widehat{\partial} \varphi(\bar{x}, \bar{y})}\big[x^{\ast} + \widehat {D}^* G (\bar{x}, \bar{y}) (y^{\ast})\big]  
\end{align}
\end{description}
provided that at least one of the conditions below holds:
\begin{description}
\item[\bf(a)] $G$ is semilocally nearly-isolated calm at $(\bar{x},\bar{y})$.

\item[\bf(b)] $G$ is locally nearly-isolated calm at $(\bar{x},\bar{y})$, and  for every $\varepsilon > 0$ there exists $\gamma > 0$ such that
\begin{align}\label{H}
\mu(x) = \inf_{y \in G(x)} \varphi (x, y) = \inf_{y \in G(x) \cap \mathbb{B} 
(\bar{y}, \varepsilon)} \varphi (x, y)\;\mbox{ for  }\;x \in \mathbb{B} 
(\bar{x},\gamma).
 \end{align}

\item[\bf(c)] The solution mapping $M : \dom G\rightrightarrows Y$ admits an upper Lipschitzian selection  at $(\bar x, \bar y)$.
\end{description}
\end{theorem}

\begin{proof} The upper estimate in (i) is taken \cite[Theorem 1]{MNY2009} formulated there in the Banach space setting, while the given proof does not require the completeness of the normed spaces in question. Let us now verify the lower estimate of regular subgradients in (ii) in all the three cases therein.\vspace*{0.05in}

\noindent {$\bullet$} {\em Case~{\rm (a)}, where $G$ is semilocally nearly-isolated calm at 
$(\bar{x},\bar{y})$ with modulus $L$.} Fix an arbitrary number $\varepsilon > 0$ and pick
$u^{\ast}$ in the right-hand side of (\ref{1722023_4}). Then there 
exists ${(x^{\ast}, y^{\ast}) \in \widehat{\partial} \varphi (\bar{x}, 
\bar{y})}$ such that $u^{\ast}-{x}^{\ast}\in \widehat {D}^* 
G(\bar{x}, \bar{y}) (y^{\ast})$, and hence $(u^{\ast}-{x}^{\ast}, - y^{\ast}) 
\in \widehat{N} ( (\bar{x}, \bar{y});\gr G)$. The inclusion ${(x^{\ast}, y^{\ast}) 
\in \widehat{\partial} \varphi (\bar{x}, \bar{y})}$ allows us to 
choose $\delta_1 > 0$ so that for every $(u,v) \in \mathbb{B} 
((\bar{x}, \bar{y}), \delta_1)$ we get
\begin{equation}\label{1712023_6}
\langle x^{\ast}, u-\bar{x}\rangle+\langle y^{\ast}, v-\bar{y} 
\rangle\le \varphi(u,v)-\varphi(\bar{x},\bar{y})+\varepsilon 
(\|u-\bar{x}\|+\|v-\bar{y}\|).
\end{equation}
By Lemma~\ref{12202023_lem}, there exists a function $s: X \times Y \rightarrow \mathbb{R}$ Fr\'echet differentiable at $(\bar x,\bar y)$ with $\nabla s(\bar{x}, \bar{y}) 
= (u^{\ast} - x^{\ast}, - y^{\ast})$ and  
\begin{equation}\label{for:3-1}
 s(u, v) \le s(\bar{x}, \bar{y}) = 0\;\mbox{ for all }\;(u, v) \in \gr G\;\mbox{ near }\;(\bar x,\bar y).
\end{equation}
Thus there exists $0<\delta_2 <\delta_1$ such that 
\begin{equation}\label{for:3-2}
 \arrowvert s(u, v) - s(\bar{x}, \bar{y}) - \langle (u^{\ast} - x^{\ast}, - 
 y^{\ast}), (u - \bar{x}, v - \bar{y}) \rangle \lvert \le \varepsilon ( \|u - 
 \bar{x}\| + \|v - \bar{y}\|)
\end{equation}
for every $(u, v) \in \mathbb{B}((\bar{x}, \bar{y}), \delta_2)$, 
Combining the latter with \eqref{for:3-1} yields 
\begin{equation}\label{for:3-3}
\langle u^{\ast}, u - \bar{x} \rangle \le \langle x^{\ast}, u - \bar{x} \rangle 
+\langle y^{\ast}, v - \bar{y} \rangle + \varepsilon (\|u - \bar{x}\| + 
\|v - \bar{y}\|)
\end{equation}
for every $(u, v) \in \mathbb{B}((\bar{x}, \bar{y}), \delta_2) \cap \gr G$. The imposed semilocal nearly-isolated calmness of $G$ at $(\bar{x},\bar{y})$ gives us $\delta_3>0$ with $\max\{\delta_3,L\delta_3\}<\delta_2$ satisfying
\begin{equation}\label{1812023_1}
 G(u) \subset \{\bar{y}\} + L \|u - \bar{x}\|{\mathbb{\overline B}}_Y \subset 
\mathbb{B} (\bar{y}, \delta_2)\;\mbox{ whenever }\;u \in \mathbb{B} (\bar{x}, \delta_3) \setminus \{\bar{ x}\}.
\end{equation}
Observe that $(u, v) \in \mathbb{B} ((\bar{x}, \bar{y}), \delta_2) 
\cap \gr G$ whenever $u \in \mathbb{B} (\bar{x}, \delta_3) \setminus \{\bar{ x}\}$ 
and $v \in G(u)$. Thus it follows from (\ref{1712023_6}), (\ref{for:3-3}), and 
(\ref{1812023_1}) that 
\begin{align*}
\langle u^{\ast}, u - \bar{x} \rangle\le\varphi (u, v) - \varphi (\bar{x}, \bar{y}) + 2 \varepsilon (L+1)
 \|u - \bar{x}\|\;\mbox{ for all }\; ~ v  \in G(u),
\end{align*}
and therefore we arrive at the estimate
\begin{align*}
\langle u^{\ast}, u - \bar{x} \rangle & \le \mu(u) - \mu(\bar{x}) + 2 
\varepsilon (L + 1) (\|u - \bar{x}\|)\;\mbox{ when }\;u \in \mathbb{B} 
(\bar{x}, \delta_3).
\end{align*}
This tells us that $u^{\ast} \in 
\widehat {\partial} \mu (\bar{x})$ since $\varepsilon>0$ was chosen arbitrarily.\vspace*{0.05in}

\noindent $\bullet$ {\em Case~{\rm(b)}, where $G$ is locally nearly-isolated calm at 
$(\bar{x},\bar{y})$.} The proof in this case is similar to (a). Indeed, by taking $\delta_1,\delta_2$ as in (a), it follows from the local nearly-isolated calmness of $G$ at $(\bar{x},\bar{y})$ with modulus $L_{0}$ that there exists $\delta_3 > 0$ with $\max\{\delta_3,\delta_3 L_{0}\}< \delta_2$  such that for all $u \in \mathbb{B} (\bar{x}, \delta_3) \setminus \{\bar{ x}\}$ we have
\begin{align}\label{16072023}
G(u) \cap \mathbb{B} (\bar{y}, \delta_3) \subset \{\bar{y}\} + L_{0} 
\|u - \bar{x}\| {\mathbb{\overline B}}_Y \subset \mathbb{B} 
 (\bar{y}, \delta_2).
 \end{align}
 Furthermore, the choice of $u\in\mathbb{B} (\bar{x}, \delta_3) \setminus 
 \{\bar{ x}\}$ and $v \in G(u) \cap \mathbb{B} (\bar{y}, \delta_3)$ yields
 $(u, v) \in \mathbb{B} ((\bar{x}, \bar{y}), \delta_2) \cap \gr G$. Combining
\eqref{1712023_6}, \eqref{for:3-3}, and \eqref{16072023} ensures that
 that for every $u \in \mathbb{B} (\bar{x}, \delta_3) \setminus \{\bar{ x}\}$
 we get the estimate
 \begin{align}\label{12202023_1}
\langle u^{\ast}, u - \bar{x} \rangle \le \varphi (u, v) - \varphi (\bar{x}, \bar{y}) + 2 \varepsilon (L_{0} + 1) 
\|u - \bar{x}\|	
\end{align}
whenever $v\in G(u)\cap \mathbb{B} (\bar{y}, \delta_3)$. On the other hand, it follows from assumption $\eqref{H}$ that there exists $\bar \delta \le \delta_3$ such that 
$$
\mu(u) = \inf_{y \in G(u)} \varphi (u, y) = \inf\limits_{y \in G(u) \cap 
\mathbb{B} (\bar{y}, \delta_3)} \varphi (u, y)\;\mbox{ for all }\; u \in \mathbb{B} 
(\bar x, \bar\delta).
$$ 
Unifying the latter with (\ref{12202023_1}) brings us to the implication
\begin{align*}
&\langle u^{\ast}, u -\bar{x} \rangle \le \mu(u) - \mu (\bar{x}) + 2 
\varepsilon (L_0+1) \|u - \bar{x}\|\;\mbox{ if }\; u \in \mathbb{B} (\bar{x}, 
\bar\delta) \setminus \{\bar{x}\} \\
& \Longrightarrow \langle u^{\ast}, u -\bar{x} \rangle\le \mu(u) - \mu 
(\bar{x}) + 2 \varepsilon (L_0 + 1) \|u - \bar{x}\|\;\mbox{ if }\; u \in 
\mathbb{B} (\bar{x},\bar\delta),
\end{align*}
 which means that $u^* \in \widehat {\partial} \mu (\bar{ x})$ verifying the lower estimate \eqref{1722023_4} in case (b).\vspace*{0.05in}

\noindent $\bullet$ {\em Case~{\rm(c)}, where the solution mapping $M : \dom G 
\rightrightarrows Y$ admits a local upper Lipschitzian selection $h$ at 
$(\bar x, \bar y)$.} In this case, we have $\mu(x)=\varphi(x,h(x))$ for all $x$ from some neighborhood $U$ of $\bar x$. As in case (a), there exist $0<\delta_3< \delta_2$ such 
that $h(\bar{ x})=\bar{ y}$, $h(x) \in M(x)$, and 
\begin{align}\label{12212023}
\|h(x)-h(\bar{ x})\|<\ell\|x-\bar{ x}\| < \delta_2\;\mbox{ for all }\;x \in \mathbb{B} 
 (\bar{x}, \delta_3) \subset U(\bar{ x})
\end{align}
with the calmness modulus $\ell$ from \eqref{calm}. Then for every $u \in \mathbb{B} (\bar{x}, \delta_3)$, we have $(u, h(u))\in \mathbb{B} ((\bar{x}, \bar{y}), \delta_2) \cap \gr G$, which implies by 
\eqref{1712023_6} and \eqref{for:3-3} that
\begin{align*}
\langle u^{\ast}, u - \bar{x} \rangle & \le \varphi (u,h(u)) - \varphi 
(\bar{x}, \bar{y}) + 2 \varepsilon (\|u - \bar{x}\| + \|h(u) - \bar{y}\|)\\
& \le \mu(u) - \mu(\bar{ x}) + 2 \varepsilon (\ell + 1) \|u - \bar{ x}\|\;\mbox{ whenever }\;u \in \mathbb{B} (\bar x, \delta_3).
\end{align*}
Since $\varepsilon > 0$ was chosen arbitrarily, we conclude that $u^{*} \in \widehat 
{\partial} \mu (\bar{ x})$, which verifies \eqref{1722023_4} and completes the proof the theorem.
\end{proof}

The next theorem provides both lower and upper estimates for the {\em singular regular subgradients} \eqref{eq:Fsin} of marginal functions defined on normed spaces. The only result on {\em upper estimates} of $\widehat\partial^\infty\mu(\bar x)$ known to us is given in \cite[Theorem~4.1]{HuxuAn}, where a different estimate is obtained in Asplund spaces via the singular subdifferential and coderivative by Mordukhovich under the metric qualification conditions (Q1) and (Q2) from Definition~\ref{mqc}. We are not familiar with any {\em lower estimate} of $\widehat\partial^\infty\mu(\bar x)$ available in the literature. 

\begin{theorem}[\bf estimates for singular regular
 subgradients]\label{thm192024} Let $\bar{y} \in M(\bar{x})$ for the solution mapping \eqref{1722023_2} associated with \eqref{1722023_1}, where $\varphi$ is finite at $\bar x$. The following upper and lower estimates for $\widehat {\partial}^{\infty}\mu(\bar{x})$ are satisfied:
\begin{description}
\item[\bf(i)] If $\varphi$ is calm at $(\bar{ x}$,$\bar{ y}$) as in \eqref{calm}, 
then we have the upper estimate
\begin{align}\label{192024_3}
\widehat {\partial}^{\infty}\mu(\bar{x}) \subset \widehat {D}^* 
G(\bar{x},\bar{y})(0).
\end{align}

\item[\bf(ii)] The lower estimate for singular regular subgradients
\begin{align}\label{192024_4}
\widehat{\partial}^{\infty} \mu (\bar{x})\supset\bigcup_{(x^*, y^*) \in \widehat{\partial}^\infty \varphi (\bar{x}, \bar{y})}\big[ 
x^* + \widehat {D}^* G(\bar{x}, \bar{y}) (y^*) \big]
\end{align}
\end{description}
holds if at least one of the conditions below is fulfilled:
\begin{description}
\item[\bf(a)] $G$ is semilocally nearly-isolated calm at $(\bar{x},\bar{y})$.
	
\item[\bf(b)] $G$ is locally nearly-isolated calm at $(\bar{x},\bar{y})$, and  for every $\varepsilon > 0$ there exists $\gamma > 0$ such that \eqref{H} is satisfied.

\item[\bf(c)] The solution mapping $M : \dom G\rightrightarrows Y$ admits a local 
upper Lipschitzian selection  at $(\bar x, \bar y)$,
\end{description}
\end{theorem}

\begin{proof}
To verify the upper estimate \eqref{192024_3} in (i), observe that 
\begin{align*}
\ep (\varphi + \delta_{\gr G}) = \ep \varphi \cap [\gr G \times \mathbb{ R}].
\end{align*}
Picking $u^*\in\widehat {\partial}^{\infty} \mu(\bar{ x})$, we get by \cite[Proposition~4.1]{HuxuAn} that $(u^*, 0) \in \widehat {\partial}^{\infty} (\varphi + \delta_{\gr G}) (\bar{x}, 
\bar{ y})$, which is equivalent to 
 $(u^*, 0, 0) \in \widehat N \left( (\bar{x}, \bar{y}, \varphi (\bar{x}, \bar{y})); 
 \ep \varphi\cap [\gr G\times \mathbb{R}] \right)$, and hence 
 $(u^*,0,0)\in \widehat N\left( (\bar{ x},\bar{ y},\varphi 
 (\bar{ x},\bar{ y})); \gr \varphi\cap [\gr G\times \mathbb{ R}] \right)$. Thus
 for every $\varepsilon>0$ there exists $\delta>0$ such that 
 \begin{align*}
\langle u^*, x - \bar x \rangle \le \varepsilon (\|x - \bar x\| + \|y - \bar{y}\| 
 + \varphi (x, y) - \varphi (\bar{ x}, \bar{ y})|)
 \end{align*}
 whenever $(x, y, \varphi (x, y)) \in (\gr \varphi \cap [\gr G \times \mathbb{R}]) 
 \cap \mathbb{B} ((\bar{x}, \bar{y}, \varphi (\bar{x}, \bar{y})), \delta)$. Employing the assumed calmness property of $\varphi$ gives us 
 $\kappa\ge 0$ ensuring that
 \begin{align}\label{1_9_2023}
\langle u^*, x - \bar x \rangle \le \varepsilon (\kappa + 1) (\|x - \bar x\| + 
\|y - \bar{ y}\|)
 \end{align}
 for all $(x, y) \in \mathbb{B} ((\bar{x}, \bar{y}), \delta)$ and $y \in G(x)$.
 Remembering that $\varepsilon > 0$ was chosen arbitrarily and using (\ref{1_9_2023}), we get the relationship
\begin{align*}
\limsup\limits_{(x, y) \xrightarrow{\gr G}(\bar{ x}, \bar{ y})} \dfrac{\langle 
 u^*, x - \bar x\rangle - \langle 0, y - \bar y\rangle}{\|x - \bar x\| + \|y - 
\bar{ y}\|}\le 0
\end{align*}
meaning that $u^{*} \in \widehat D^* G(\bar{x}, \bar{y}) (0)$, i.e., $\widehat {\partial}^{\infty} \mu(\bar{x}) \subset \widehat {D}^* G(\bar{x}, 
\bar{y}) (0)$. This justifies the claimed upper subdifferential estimate 
\eqref{192024_3}. 
	
Now we derive the lower subdifferential estimate in (ii). Fix an arbitrary number $\varepsilon > 0$, pick $u^{\ast}$ from the right-hand side of \eqref{192024_4}, and ${(x^{\ast}, y^{\ast}) \in{\widehat {\partial}^\infty} \varphi (\bar{x}, \bar{y})}$ with $u^{\ast} - 
{x}^{\ast} \in \widehat {D}^* G(\bar{x},\bar{y}) (y^{\ast})$, i.e.,
$(u^{\ast} - {x}^{\ast}, -y^{\ast})\in \widehat{N} ( (\bar{x},\bar{y});\gr G)$. It follows from ${(x^{\ast},y^{\ast})\in {\widehat {\partial}^\infty} \varphi 
(\bar{x}, \bar{y})}$ that
$$
(x^{\ast},y^{\ast},0)\in \widehat N\big((\bar{ x},\bar{ y}), \varphi(\bar{ x}, 
\bar{ y});\ep\varphi\big),
$$
which yields the existence of $\delta_1>0$ such that for every triple $(u, v, \alpha) \in 
\mathbb{B} (((\bar{x}, \bar{y}), \varphi (\bar{ x}, \bar{ y})), \delta_1) 
\cap \ep \varphi$ we have the estimate
\begin{align}\label{29062023_1}
\langle x^{\ast}, u - \bar{x} \rangle + \langle y^{\ast}, v - \bar{y} \rangle 
\le \varepsilon (\|u - \bar{x}\| + \|v - \bar{y}\| + |\alpha - \varphi (\bar{x}, 
\bar{ y})|).
\end{align}
Lemma~\ref{12202023_lem} tells us that there exists a function $s: X \times Y \rightarrow \mathbb{R}$, which is Fr\'echet differentiable at $(\bar{x},\bar y)$ while satisfying $\nabla s(\bar{x}, \bar{y}) = (u^{\ast}-x^{\ast}, -y^{\ast})$ and  \eqref{for:3-1}. This implies in turn that there exists $0 < \delta_{2} < \delta_{1}$ ensuring the fulfillment of \eqref{for:3-2} and \eqref{for:3-3}. Therefore, for every 
$(u, v) \in \mathbb{B} ((\bar{x}, \bar{y}), \delta_2) \cap \gr G$ we have
\begin{equation}\label{29062023_3}
\langle u^{\ast}, u - \bar{x} \rangle \le \langle x^{\ast}, u - \bar{x} \rangle + 
\langle y^{\ast},  v - \bar{y} \rangle + \varepsilon (\|u - \bar{x}\| + \|v - 
\bar{y} \|).
\end{equation}

Now we are ready to verify \eqref{192024_4} in all the three cases listed in (ii).\vspace*{0.05in}

\noindent{$\bullet$} {\em Case~{\rm(a)}, where $G$ is semilocally nearly-isolated calm at 
$(\bar{x},\bar{y})$ with modulus $L$.} Taking into account the discussions above and the imposed nearly-isolated calmness condition, we find $\delta_3 > 0$ with $\max\{\delta_3,L\delta_3\}<\delta_2$ such that semilocal calmness condition \eqref{1812023_1} is satisfied. Observing that 
$(u, v) \in \mathbb{B} ((\bar{x}, \bar{y}), \delta_2) \cap \gr G$ for each vectors $u \in \mathbb{B} (\bar{x}, \delta_3) \setminus 
\{\bar{x}\}$ and $v \in G(u) \cap \mathbb{B} (\bar{y}, \delta_3)$ allows us to deduce from \eqref{1812023_1}
\eqref{29062023_1}, and \eqref{29062023_3} that 
\begin{align}\label{29062023_5}
\langle u^{\ast}, u - \bar{x} \rangle \le 2 \varepsilon (\|u - \bar{x}\| + 
\|v - \bar{y}\| + |\alpha - \varphi (\bar{x}, \bar{y})|)
\end{align}
for every $(u, v, \alpha) \in [(\mathbb{B} (\bar{x}, \delta_3) \setminus 
\{\bar{x}\} \times \mathbb{B} (\bar{ y}, \delta_2)) \cap \gr G \times
\mathbb{B}_\mathbb{R} (\mu (\bar{x}), \delta_1)] \cap \ep \varphi$.  
Then we claim the fulfillment of the estimate
\begin{align}\label{est-sing}
\langle u^{\ast}, u - \bar{x} \rangle \le 2 \varepsilon (L + 1) (\|u - \bar{x}\| 
 + |\alpha-\mu(\bar{ x})|)
\end{align} 
for all $(u, \alpha) \in B((\bar{x}, \mu(\bar{x})); \delta_3) \cap\ep\mu$. To verify \eqref{est-sing}, take $\alpha>\mu(u)$ and a sequence of $\{v_k\}$ satisfying
$$
\varphi (u, v_k) \xrightarrow{\{v_k\}\subset G(u)} \mu(u) \quad \text{as} 
\quad k\to \infty.
$$ 
Then $\alpha \ge \varphi (u, v_k)$ for $k$ sufficiently large, and thus
$$
(u, v_k, \alpha) \in [( (\mathbb{B} (\bar{x}, \delta_3) \setminus 
\{\bar{ x}\} ) \times \mathbb{B} (\bar{ y}, \delta_2)) \cap \gr G \times 
\mathbb{B}_\mathbb{R} (\mu (\bar{ x}), \delta_1)] \cap \ep \varphi.
$$ 
If follows from (\ref{1812023_1}) and (\ref{29062023_5}) that
\begin{align*}
\langle u^{\ast}, u -\bar{x}\rangle&\le 2\varepsilon(\|u-\bar{x}\| + \|v_k 
 - \bar{y}\|+| \alpha-\mu(\bar{ x})|)\\
 & \le 2 \varepsilon (L+1)(\|u-\bar{x}\| + |\alpha - \mu(\bar{ x})|)
\end{align*} 
for large $k$. If $\alpha = \mu(u)$ for all $\gamma > 0$ satisfying $\alpha + \gamma \in 
\mathbb{B}_{\mathbb{ R}} (\mu(\bar{ x}), \delta_3)$, then $\alpha + 
\gamma > \mu (u)$, and the same device as above brings us to
\begin{align*}
\langle u^{\ast}, u - \bar{x} \rangle \le 2 \varepsilon (L + 1) (\|u - \bar{x}\| 
+ |\alpha + \gamma - \mu (\bar{ x})|).
\end{align*}
Since $\gamma>0$ was chosen arbitrarily small, this justifies the claimed estimate \eqref{est-sing}. The arbitrary choice of $\varepsilon > 0$ ensures that $(u^*,0) \in \widehat N 
((\bar{ x}, \mu(\bar{ x})); \ep \mu)$, and thus we get $u^{\ast} \in \widehat 
{\partial}^\infty \mu(\bar{x})$ justifying \eqref{192024_4}.\vspace*{0.05in}

\noindent $\bullet$ {\em Case~{\rm(b)}, where $G$ is locally nearly-isolated calm at 
$(\bar{x},\bar{y})$ under the fulfillment of \eqref{H}.}  The proof of \eqref{192024_4} in this case is similar to the above case (a) with taking into account the arguments in case (b) of Theorem~\ref{thm1}.\vspace*{0.05in}

\noindent $\bullet$ {\em Case~{\rm(c)}, where the solution mapping $M : \dom G 
\rightrightarrows Y$ admits an upper Lipschitzian selection at $(\bar x, \bar y)$}. Let $h$ be such a selection with modulus $\ell>0$ in \eqref{calm}. We have $\mu (x) = \varphi(x,h(x))$ for all $x \in U$ for a neighborhood $U$ of $\bar{x}$. As in the proof of case (c) in Theorem~\ref{thm1}, find $0 < \delta_3< \delta_2$ such that $h(\bar{ x}) = \bar{y}$, $h(x) \in M(x)$, and condition \eqref{12212023} is satisfied. Furthermore, $(u, h(u)) \in 
\mathbb{B} ((\bar{x}, \bar{y}), \delta_2) \cap \gr G$ for every $u \in \mathbb{B} (\bar{x}, \delta_3)$, and thus it follows from the estimates in
\eqref{29062023_1} and \eqref{29062023_3} that
\begin{align*}
\langle u^{\ast}, u - \bar{x} \rangle \le 2 \varepsilon (\|u - \bar{x}\| + 
\|h(u) - \bar{y}\| + |\alpha - \varphi (\bar{x}, \bar{y})|)
\end{align*}
for all $(u, h(u), \alpha) \in [(\mathbb{B} ((\bar{x}, \delta_3) \times 
\mathbb{B} (\bar{y}, \delta_2)) \cap \gr G \times \mathbb{B}_{\mathbb{R}} 
(\mu(\bar{ x}), \delta_1) ]\cap \ep \varphi$. By \eqref{12212023} and $\mu(u) = 
\varphi (u, h(u))$ for all $u \in\mathbb{B} (\bar{x}, \delta_3)$, we get
\begin{align*}
\langle u^{\ast}, u - \bar{x} \rangle\le 2\varepsilon(\ell+1)(\|u-\bar{x}\|+|\alpha-\mu(\bar{ x})|)
\end{align*}
whenever $(u, \alpha) \in [\mathbb{B} (\bar{x}, \delta_3) \times (\mathbb{B}_{
\mathbb{ R}} (\mu (\bar{ x}), \delta_3)] \cap \ep \mu$. This tells us, by the arbitrary choice of $\varepsilon>0$, that $(u^{*}, 0) \in \widehat N ((\bar{x}, 
\mu (\bar{x})); \ep \mu)$, and hence $u^{\ast} \in \widehat {\partial}^{\infty} 
\mu (\bar{x})$, which completes the proof of theorem. 
\end{proof}

We conclude this section with an example illustrating the subdifferential estimates obtained in Theorems~\ref{thm1} and \ref{thm192024}.

\begin{example}[\bf illustrating subdifferential estimates for marginal functions]
{\rm Consider the marginal function $\mu(\cdot)$ 
in (\ref{1722023_1}) with $X=Y=\mathbb R$, 
$$ \varphi(x,y) := \lvert x \rvert - y,
~~~ {\rm and} ~~~
G(x) = \left\{\begin{array}{ll}
\{x\} & {\rm if} ~ x > 0,\\
(-\infty,0] & {\rm if} ~ x = 0,\\
\emptyset,& {\rm otherwise}.
\end{array}
\right.
$$
Then we have
$$ 
\mu(x) 
= \left\{\begin{array}{ll}
0, & {\rm if} ~ x \ge 0,\\
+\infty, & {\rm otherwise},
\end{array}
\right.\text{and}\,\,\,\,
M(x) 
= \left\{\begin{array}{ll}
\{x\}, & {\rm if} ~ x \ge 0,\\
\emptyset, & {\rm otherwise}.
\end{array}
\right.
$$ 

Here $G$ is locally nearly-isolated calm at $(\bar{x}, \bar{y}) := (0, 0)$ 
for every $L>0$ and assumption \eqref{H} is also valid, 
\begin{align*}
\widehat{D}^*G(0,0)(y^*) = \left\{\begin{array}{ll}
\emptyset & {\rm if} ~ y^* > 0,\\
(-\infty,y^*] & {\rm if} ~ y^* \le 0,
\end{array}
\right.
\end{align*}
$\widehat{\partial} \varphi (0, 0) = [-1, 1] \times \{-1\}$, $\widehat{
\partial}^\infty \varphi (0, 0) = \{(0, 0)\}$, and $\widehat{\partial}^{+} 
\varphi (0, 0) = \emptyset$. Due to the latter fact, the upper estimate of $\widehat\partial\mu(0)$ in Theorem~\ref{thm1}(i) cannot be applied. On the other hand, Theorem~\ref{thm1}(ii) gives us a lower 
estimate of $\widehat{\partial}\mu(0)$ by
\begin{align*}
\widehat{\partial} \mu (0)\supset\bigcup_{(x^{\ast}, y^{\ast}) \in \widehat{\partial} \varphi(0, 0)} 
\big[x^{\ast} + \widehat {D}^* G (0, 0) (y^{\ast})\big]=(-\infty,0].
\end{align*}
Furthermore, it follows from inclusions \eqref{192024_3} and \eqref{192024_4} of 
Theorem~\ref{thm192024} that
\begin{align*}
\widehat{\partial}^\infty \mu (0)=\widehat {D}^* G (0, 0) (0)=(-\infty,0].
\end{align*}}
\end{example}\vspace*{-0.2in}

\section{Exact Calculations of Regular and Singular Regular Subdifferentials of Marginal Functions}\label{sec:4}

The results of this section provide sufficient conditions ensuring {\em exact formulas} (i.e., {\em equalities)} for calculating both regular and singular regular subdifferentials of marginal functions. We present two theorems of this type: one for marginal functions with Fr\'echet differentiable costs, and the other one for fully nonsmooth data under additional qualification conditions. As in the previous section, all the spaces in question are {\em arbitrarily normed}.\vspace*{0.05in} 

The first theorem concerns the case where the cost function in \eqref{1722023_1} is Fr\'echet differentiable. Formula \eqref{1822023_2} for the regular subdifferential can be found in \cite[Theorem~2]{MNY2009} under condition (a) formulated in Theorem~\ref{thm192024}. The other results of the following theorem are, to the best of our knowledge, new. 

\begin{theorem}[\bf marginal functions with differentiable costs]\label{thm2} Let $\mu(\cdot)$ \\
 from \eqref{1722023_1} be finite at 
 some $\bar{x} \in \dom M$, and let $\bar{y} \in M(\bar{x})$. Suppose 
 that $\varphi$ is Fr\'echet differentiable at $(\bar{x}, \bar{y})$. Then we have the calculation formulas
 \begin{align}
& \widehat{\partial} \mu (\bar{x}) = x^{\ast} + \widehat {D}^* G (\bar{x}, 
\bar{y}) (y^{\ast})\;\mbox{ with }\;(x^{\ast}, y^{\ast}) := \nabla \varphi(\bar{x},\bar{y}), \label{1822023_2} \\
& \widehat{\partial}^\infty \mu (\bar{x}) =  \widehat {D}^* G 
(\bar{x}, \bar{y}) (0) \label{192024_2}
\end{align}
provided that at least one of the conditions {\rm(a)--(c)} of Theorem~{\rm\ref{thm192024}} is satisfied.
\end{theorem}

\begin{proof}
 Since $\varphi$ is Fr\'echet differentiable at $(\bar{x}, \bar{y})$, we have
 $$
 \widehat{\partial}^{+} \varphi (\bar{x}, \bar{y}) = \widehat{\partial} 
 \varphi (\bar{x}, \bar{y}) = \{\nabla \varphi (\bar{x}, \bar{y})\} =\big\{(x^{\ast}, 
 y^{\ast})\big\}.
 $$Therefore, equality (\ref{1822023_2}) follows from 
 \eqref{1722023_3} and \eqref{1722023_4} of Theorem~\ref{thm1}, while 
 equality \eqref{192024_2} is a consequence of \eqref{192024_3} and \eqref{192024_4} in
 Theorem \ref{thm192024}. 
\end{proof}

The following example shows that Theorem~\ref{thm2} can be applied to ensure the equalities in \eqref{1822023_2} and \eqref{192024_2} under the near-isolated calmness of $G$, while the solution mapping $M$ fails to admit an upper Lipschitzian selection,

\begin{example}[\bf equalities under nearly-isolated calmness]\label{exa-icalm}
 {\rm Let $X=Y=\mathbb{ R}$ and $(\bar{ x},\bar{ y})=(0,0)$ in the setting of Theorem~{\rm\ref{thm2}} with
 $$
 \varphi(x,y):=(x-y^2)^2
 ~~~ {\rm and} ~~~
G(x): = \left\{\begin{array}{ll}
\mathbb{R} & {\rm if} ~ x = 0,\\
\emptyset & {\rm otherwise}.
\end{array}
\right.
$$
Then we can directly calculate that
$$
\mu(x)=\left\{\begin{array}{ll}
0 & {\rm if} ~ x = 0,\\
\infty & {\rm otherwise}
\end{array}
\right.
~~~{\rm and} ~~~
M(x) = \left\{\begin{array}{ll}
\{0\} & {\rm if} ~ x = 0,\\
\emptyset & {\rm otherwise}.
\end{array}
\right.
$$
Therefore, the solution mapping $M(\cdot)$ does not admit an upper Lipschitzian 
selection at $(0, 0)$ telling us that \cite[Theorem~2]{MNY2009} cannot be applied. 
Nevertheless, since $G$ is (semilocally) nearly-isolated calm at $(0, 0)$, we deduce the claimed equalities \eqref{1822023_2} and \eqref{192024_2} from
Theorem~\ref{thm2} under the nearly-isolated calmness.}
\end{example}

The next example illustrates that neither of the assumptions of Theorem~\ref{thm2} is {\em necessary} for the subdifferential equalities therein.

\begin{example}[\bf equalities without nearly-isolated calmness]\label{exa-calm1} {\rm  Let $X=Y=\mathbb{ R}$ and $(\bar{ x},\bar{ y})=(0,0)$ in the setting of Theorem~{\rm\ref{thm2}} with
$$
 \varphi(x,y):=(x-y^2)^2
 ~~~ {\rm and} ~~~
G(x):=\mathbb R. 
$$
This example is taken from \cite[Example~2(ii)]{MNY2009}, where it is shown that the solution mapping $M$ does not admit an upper Lipschitzian selection, while the regular subdifferential equality \eqref{1822023_2} holds. In this case, it is easy to see that the constraint mapping $G$ fails to be nearly-isolated calm at $(0,0)$, and hence Theorem~\ref{thm2} cannot be applied. Nevertheless, we have here that
\begin{align*}
\widehat\partial^\infty \mu(0) =\widehat\partial \mu(0) = \{0\}, ~~ \nabla \varphi(0,0)=\{(0,0)\}, ~~ 
\widehat D^* G(0,0) (0) = \{0\},
\end{align*}
and thus both equalities \eqref{1822023_2} and \eqref{192024_2} are satisfied.}
\end{example}

Now we establish new calculation formulas for both regular and singular subdifferentials of marginal functions with {\em nondifferentiable} costs. Furnishing this involves a {\em metric qualification condition} on the data of \eqref{1722023_1}.

\begin{theorem}[\bf marginal functions with nondifferentiable costs]\label{thm5}
 Let $\mu(\cdot)$ the marginal function from \eqref{1722023_1} be finite at 
 some $\bar{x}\in\dom M$, and let $\bar{y}\in M(\bar{x})$. 
Suppose that the following conditions are fulfilled:
\begin{description}
\item[($\bf\mathcal{A}_1$)] The sets ${\rm epi}\,\varphi$ and $\gr G \times 
\mathbb{R}$ satisfy the metric qualification condition \eqref{Q2} for $\partial^{\textasteriskcentered} = \widehat{\partial}$ 
at $(\bar{x}, \bar{y},\varphi(\bar{x}, \bar{y}))$.
 
\item[($\bf\mathcal{A}_2$)] At least one of the conditions {\rm(a)--(c)} of Theorem~{\rm\ref{thm192024}} holds.
\end{description}
Then we have the  subdifferential calculation formulas
\begin{equation}\label{28062023}
{\widehat {\partial}}\mu(\bar{x})=\bigcup_{(x^{\ast},y^{\ast})\in {\widehat {\partial}} 
\varphi(\bar{x},\bar{y})}\big[x^{\ast} + \widehat {D}^* G(\bar{x}, \bar{y}) 
(y^{\ast})\big],
\end{equation}
\begin{equation}\label{29062023}
{\widehat {\partial}^\infty} \mu(\bar{x}) = \bigcup_{(x^{\ast}, y^{\ast}) 
\in {\widehat {\partial}^\infty} \varphi (\bar{x}, \bar{y})}\big[x^{\ast} + 
\widehat {D}^* G(\bar{x},\bar{y})(y^{\ast})\big].
\end{equation}
\end{theorem}

\begin{proof}  
Denoting $C :=\ep\varphi$ and $D:= \gr G\times \mathbb{R}$, we intend to show that 
\begin{align}\label{12202023_2}
(x^{*},\beta) \in \widehat{N}\big((\bar{x}, \mu(\bar{x}); \ep \mu\big) 
\Longrightarrow (x^{*}, 0, \beta) \in \widehat{N}\big((\bar{x}, \bar{y}, 
\varphi (\bar{x}, \bar{y})); C \cap D\big).
\end{align}
Indeed, it follows from Lemma~\ref{12202023_lem} that there exists a 
Fr\'echet differentiable function $s: X \times \mathbb{ R} \rightarrow 
\mathbb{R}$ such that $\nabla s (\bar{x}, \mu(\bar{x})) = 
(x^{\ast}, \beta)$ and  
\begin{equation*}
s(x,r) \le s(\bar{x}, \mu (\bar{x})) = 0\;\mbox{ whenever }\;(x, r) \in \text{epi}\, \mu.
\end{equation*}
Fix $\varepsilon > 0$ and use the Fr\'echet differentiability of $s(\cdot)$ to find
$\delta_1>0$ such that 
\begin{align*}
& |s(x,r) - \langle(x^{\ast}, \beta), (x - \bar{x}, r - \mu(\bar{x})) \rangle \lvert \le \varepsilon (\|x - \bar{x}\| + |r - \mu(\bar{x})|) \notag \\
& \Longrightarrow \, \langle(x^{\ast}, \beta), (x - \bar{x}, r - \mu(\bar{x})) 
\rangle \le \varepsilon (\|x - \bar{x}\| + |r - \mu(\bar{x})|) 
\end{align*}
for every $(x, r) \in\mathbb{B}((\bar{x}, \mu(\bar{x})), \delta_1)$. This yields the estimate
\begin{align*}
\langle x^{\ast}, x - \bar{x} \rangle & \le \langle - \beta,  r - \mu(\bar{x}) 
\rangle + \varepsilon (\|x - \bar{x}\| + |r - \mu(\bar{x})|)
\end{align*}
whenever $(x, r) \in \mathbb{B} ((\bar{x}, \mu(\bar{x})),\delta_1) \cap \ep\mu$. Since $\mu(\bar{ x})=\varphi(\bar{ x},\bar{ y})$ and $\mu(x) 
 \le \varphi(x,y)\le r$ for every $(x, y, r) \in \mathbb{B} ((\bar{x}, \bar{y}, 
 \varphi(\bar{x},\bar{y})), \delta_1)\cap C\cap D$, we have
 \begin{align*}
\langle x^{\ast}, x - \bar{x} \rangle \le & \, \langle - \beta, r - \varphi 
(\bar{x}, \bar{y}) \rangle + \varepsilon (\|x - \bar{x}\| + \|y - \bar y\| + 
|r - \varphi (\bar{x}, \bar{y})|) \notag \\
&\mbox{for all }\;(x, y, r) \in \mathbb{B} ((\bar{x}, \bar{y},\varphi(\bar{x}, 
\bar{y})), \delta_1)\cap C\cap D,  
\end{align*}
which therefore justifies the implication in \eqref{12202023_2}.
	
To verify \eqref{28062023}, pick $p^{*} \in \widehat{\partial} 
\mu (\bar{x})$ and get $(p^{*}, - 1) \in \widehat{N} (\bar{x}, \mu(\bar{x}); 
\ep\mu)$. By Lemma~\ref{12202023_lem2}, there exist  $(x^*, \beta) 
\in \widehat{\partial} d(\bar x, \mu(\bar x), \ep\mu)$ and $\lambda> 0$ 
such that $p^* = \lambda x^*$ and $-1=\lambda\beta$, where
$\lambda \ge 1$ by Lemma~\ref{lem2722023_1}. Since $\bar{y} \in M(\bar{x})$, we claim that
\begin{align*}
(x^*, \beta) \in \widehat{\partial} d(\bar x, \mu(\bar x), \text{epi}\,\mu) 
\Longrightarrow (x^{*}, 0, \beta) \in \widehat{\partial} d((\bar{x}, \bar{y}, 
\varphi (\bar{x}, \bar{y})), C \cap D).
\end{align*}
Indeed, it follows from Lemma~\ref{lem2722023_1} that
\begin{align*}
(x^*, \beta) \in \widehat{\partial} d(\bar x, \mu(\bar x), \ep \mu) = 
\mathbb{B}_{X^{*} \times \mathbb{R}} \cap \widehat{N} ((\bar{x}, \mu 
(\bar{x})); \ep\mu),
\end{align*}
which implies that $\| (x^*, \beta)\| \le 1$ and $(x^*, \beta) \in \widehat{N} 
((\bar{x}, \mu (\bar{x})); \text{epi}\, \mu)$. By using \eqref{12202023_2}, 
we get $(x^*, 0, \beta) \in \widehat{N} ((\bar{x}, \bar{y}, \varphi 
(\bar{x}, \bar{y})); C \cap D)$ and $\|(x^*, 0, \beta)\| \le 1$. Employing again 
Lemma~\ref{lem2722023_1} tells us that
\begin{align}\label{12202023_3}
 (x^*, 0, \beta) \in \widehat{\partial} d((\bar{x}, \bar{y}, \varphi 
 (\bar{x}, \bar{y})), C \cap D).
\end{align}
On the other hand, assumption ($\mathcal{A}_1$) ensures the subdifferential estimate
\begin{align*}
\widehat{\partial} d((\bar{x}, \bar{y}, \varphi(\bar{x}, \bar{y})), C \cap D) & \subset a\big[\widehat {\partial} d((\bar{x}, \bar{y}, \varphi (\bar{x}, 
\bar{y})), C) + \widehat {\partial} d((\bar{x}, \bar{y}, \varphi (\bar{x}, 
\bar{y})), D)\big],
\end{align*}
which yields in turn together with \eqref{12202023_3} the existence of the triples $(z^*, 
y^*, \gamma) \in \widehat {\partial} d ((\bar{x}, \bar{y}, \varphi 
(\bar{x}, \bar{y})), \ep \varphi)$ and $(u^*, v^*) \in \widehat{\partial} 
d((\bar{x}, \bar{y}); \gr G)$ satisfying 
\begin{align*}
&\hspace{2.1cm}\big[x^* = a (z^*+u^*), ~~ 0 = a(y^*+v^*), ~~ \beta = a 
\gamma\big] \\
& \Longrightarrow\;\big[p^* = \lambda x^* = \lambda a(z^* + u^*), ~~ 
 \lambda ay^*  = - \lambda a v^*, ~~ -1 = \lambda \beta = a \lambda 
 \gamma\big].
\end{align*}
Therefore, $(\lambda a z^*, \lambda a y^*) \in \widehat{\partial} \varphi 
(\bar{x}, \bar{y})$, $p^{*} \in \lambda a z^{*} + \widehat {D}^* 
G (\bar{x},\bar{y})(\lambda a y^{*})$, and so
\begin{align*}
\widehat{\partial} \mu (\bar{x}) \subset \bigcup_{(x^*, y^*) \in 
\widehat{\partial} \varphi (\bar{x}, \bar{y})}\big[x^* + \widehat {D}^* G 
(\bar{x}, \bar{y}) (y^*)\big].
\end{align*}
Unifying the latter with  the opposite inclusion \eqref{1722023_4} in 
Theorem~\ref{thm1} justifies the claimed equality \eqref{28062023} for regular subgradients of marginal functions.

It remains to verify equality \eqref{29062023} for singular regular subgradients. In fact, the upper singular subdifferential estimate
\begin{align*} 
\widehat{\partial}^{\infty} \mu (\bar{x}) \subset \bigcup_{(x^*, y^*)
\in \widehat{\partial}^\infty \varphi (\bar{x}, \bar{y})} x^* + \widehat {D}^* 
 G(\bar{x},\bar{y})(y^*)
\end{align*}
can be justified by using the same device as above for $\widehat\partial\mu(\bar x)$, while the lower estimate for $\widehat{\partial}^{\infty}\mu(\bar x)$ is \eqref{192024_4} obtained in 
Theorem~\ref{thm192024}.
\end{proof}

Finally in this section, we construct an example of a {\em convex} marginal function for which the results of Theorem~\ref{thm5} give us the exact calculation formulas for both regular and singular regular subdifferentials, while known results in this direction (obtained for the convex subdifferential of marginal functions) are not applicable. Although regular subgradients and normals for convex functions and sets reduce to the classical ones of convex analysis, we will keep the hat-notation for them for notation consistency. 

\begin{example} [\bf subgradient calculations for convex marginal functions]\label{exa-convex} $\,$ {\rm Let $X = Y = \mathbb{ R}^2$ and $(\bar{x}, \bar{y}) = (0_{\mathbb{R}^2}, 
 0_{\mathbb{ R}^2})$. Define 
$$
f(y): = \left\{\begin{array}{ll}
0 & {\rm if} ~ y = 0_{\mathbb{ R}^2},\\
\infty & {\rm otherwise}
\end{array}
\right.
$$
and consider the marginal function $\mu(\cdot)$ in 
\eqref{1722023_1} with the data
$$
G(x): = \left\{\begin{array}{ll}
\mathbb{R}_{+} \times \{0\} & {\rm if} ~ x = 
0_{\mathbb{R}^2}, \\
\emptyset & {\rm otherwise}
\end{array}
\right.
~~~ {\rm and} ~~~\varphi (x, y) := f(y),\quad (x, y) \in X \times Y. 
$$
Since $\dom\varphi=\mathbb{ R}^2\times\{0_{\mathbb{R}^{2}}\}$, 
$\gr G = \{0_{\mathbb{ R}^2}\} \times \mathbb{R}_{+} \times \{0\}$ with
$\inte(\gr G) = \emptyset$, and $\varphi$ is discontinuous at any point 
$(x_0, y_0) \in \gr G$, the result of \cite[Theorem~4.2]{AnYen2015} cannot be 
applied. We also are not able to use \cite[Theorems~3.3 and 3.4]{AnYao2016} due to
$$
\dom\varphi - \gr G = \mathbb{ R}^2 \times \mathbb{R}_{-} \times \{0\} 
\ne X \times Y.
$$ 
Moreover, observe that \cite[Theorem~4.56]{MN2022} is not applicable because $\varphi$ is discontinuous, the set $\mathbb{R}_{+} (\dom \varphi - \gr G)$ is not a closed subspace of 
$X \times Y$, and we have $\ri(\dom\varphi) \cap \ri (\gr G) = \emptyset$.

It is easy to see that the solution mapping \eqref{1722023_2} in this problem is 
\begin{align*}
M(x) = \left\{\begin{array}{ll}
\{0_{\mathbb{ R}^2}\} & {\rm if} ~ x = 0_{\mathbb{ R}^2},\\
\emptyset & {\rm otherwise},
\end{array}
\right.
\end{align*}
which does not admit a local upper Lipschitzian selection at 
$(0_{\mathbb{R}^2}, 0_{\mathbb{ R}^2})$. Thus we are not familiar with any result in the literature that can be applied to calculate the convex subdifferential of the marginal function under consideration. 

On the other hand, the direct calculations show that 
\begin{align*}
& \text{
 $\ep \varphi = \mathbb{R}^{2} \times \{0_{\mathbb{ R}^2}\} \times 
 \mathbb{ R}_{+}$},\\
 & \widehat N ((0_{\mathbb{R}^2}, 0_{\mathbb{R}^2}, 0); \gr G \times 
 \mathbb{R}) = \mathbb{R}^{2} \times \mathbb{R}_{-} \times \mathbb{R} 
 \times \{0\},\\
 & \widehat N ((0_{\mathbb{R}^2}, 0_{\mathbb{R}^2}, 0); \ep \varphi) = 
 \{0_{\mathbb{R}^2}\} \times \mathbb{R}^2 \times \mathbb{R}_{-},\\
 & \widehat N ((0_{\mathbb{R}^2}, 0_{\mathbb{R}^2}, 0); \ep \varphi \cap 
 (\gr G\times \mathbb{R})) = \widehat N ((0_{\mathbb{R}^2}, 
 0_{\mathbb{R}^2}, 0); \{0_{\mathbb{R}^2}\} \times 
 \{0_{\mathbb{R}^2}\} \times \mathbb{R}_+) \\
 & = \mathbb{ R}^{2} \times \mathbb{R}^{2} \times \mathbb{R}_{-}.
\end{align*}
Observe furthermore the representation
\begin{align*}
& \hspace{1.5cm} 
\widehat N ((0_{\mathbb{R}^2}, 0_{\mathbb{ R}^2}, 0); 
\ep \varphi \cap (\gr G \times \mathbb{R})) \\ 
& = \widehat N ((0_{\mathbb{R}^2}, 0_{\mathbb{R}^2}, 0); \ep \varphi) 
+ \widehat N ((0_{\mathbb{ R}^2}, 0_{\mathbb{ R}^2}, 0); \gr G \times 
\mathbb{ R}),
\end{align*}
which ensures by \cite[Theorem~3.3]{HuxuAn} that assumption $({\cal A}_1)$ is satisfied. Since 
$G$ clearly enjoys the semilocal nearly-isolated calmness property at 
$(0_{\mathbb{R}^2}, 0_{\mathbb{R}^2})$ for every modulus $L>0$, 
assumption ($\mathcal{A}_2$) holds too. Hence the equalities in \eqref{28062023} 
and \eqref{29062023} follow from Theorem~\ref{thm5}. We can also confirm this by the direct calculations, which show that
\begin{align*}
\mu(x) = \inf\big\{f(y)\;\big|\;y\in G(x)\big\} = \left\{\begin{array}{ll}
0 & {\rm if} ~ x = 0_{\mathbb{R}^2}, \\
\infty & {\rm otherwise}
\end{array}
\right.
\end{align*}
and $\widehat {\partial}^{\infty} \mu 
(0_{\mathbb{ R}^2}) = \widehat {\partial} \mu (0_{\mathbb{ R}^2}) = 
\mathbb{ R}^2$ due to 
$$
\widehat N (0_{\mathbb{R}^2}; \epi \mu) = \widehat N (0_{\mathbb{R}^2}; 
\{0_{\mathbb{ R}^2}\} \times \mathbb{R}_{+}) = \mathbb{ R}^2 \times 
\mathbb{R}_{-},$$ 
$$\widehat {\partial}^{\infty} \varphi (0_{\mathbb{ R}^2}, 0_{\mathbb{ R}^2}) 
= \widehat {\partial} \varphi (0_{\mathbb{R}^2}, 0_{\mathbb{R}^2}) = 
\{0_{\mathbb{ R}^2}\} \times \mathbb{ R}^2,\;\mbox{ and}$$ 
\begin{align*}
\widehat {D}^* G (0_{\mathbb{ R}^2}, 0_{\mathbb{ R}^2}) (y^{*}_1,y^{*}_2) = 
\left\{\begin{array}{ll}
\mathbb{R}^{2} & {\rm if} ~ y^{*}_{1} \ge 0,\\
\emptyset & {\rm otherwise}.
\end{array}
\right.
\end{align*}}
\end{example}\vspace*{-0.15in}

\section{Applications to Variational Calculus}\label{sec:5}

This section develops new {\em calculus rules as equalities} for regular subgradients and their singular counterparts in general {\em normed spaces}. Needless to saying that such a {\em variational subdifferential calculus} plays a crucial role in applications of variational analysis to structured problems of optimization and related areas. 

The obtained calculus rules are mainly derived from the calculation results for marginal functions established in Section~\ref{sec:4} with the systematical usage of the {\em calmness} and {\em metric qualification} conditions. The calculus rules for both regular and singular regular subgradients obtained below under the metric qualification conditions seem to be completely new, while some of the results for regular subgradients involving calmness \eqref{calm} can be found in \cite{MNY2006} under the more restrictive local Lipschitz continuity. Note that the chain rules of the inclusion type derived in \cite{HuxuAn} under metric qualification conditions in Asplund spaces are fully different from ours while providing upper estimates of regular and singular regular subdifferentials in terms of the limiting ones by Mordukhovich.\vspace*{0.03in}

When the constraint mapping $G := g: X \to Y$ in (\ref{1722023_1}) is {\em single-valued}, the marginal function $\mu(\cdot)$ reduces to the (generalized) composition
\begin{align}\label{compo12232023}
 \mu(x)=\left(\varphi\circ g\right)(x):=\varphi\big(x,g(x)\big),\quad x\in X,
\end{align}
which is the standard one $\varphi(g(x))$ when $\varphi$ does not depend 
on $x$.\vspace*{0.03in}

First we present chain rules for regular and singular regular subdifferentials of generalized compositions \eqref{compo12232023}.
		
\begin{theorem}[\bf subdifferentiation of generalized compositions]\label{Theo12232023}
 Let $\varphi: X \times Y \to \mathbb{\bar R}$ be finite at $(\bar x, \bar y)$ 
 with $\bar y := g(\bar x)$, and let $g: X \to  Y$ be calm at $\bar{x}$. 
 Assume that $\varphi$ is Fr\'echet differentiable at $(\bar x,\bar y)$ with  
 $\nabla\varphi(\bar x,\bar y):=(\nabla_x\varphi(\bar x,\bar y), \nabla_y 
 \varphi(\bar x,\bar y))$. Then the following chain rules hold for generalized compositions \eqref{compo12232023}:
\begin{equation*}
\begin{array}{ll}
\widehat {\partial} \left(\varphi \circ g \right) (\bar{x})&= \nabla_x 
\varphi (\bar{x}, \bar{y}) + \widehat D^* g(\bar{x}) \left(\nabla_y \varphi 
(\bar{x}, \bar{y}) \right)\\
&= \nabla_x \varphi (\bar{x}, \bar{y}) + \widehat{\partial} 
\langle \nabla_y \varphi (\bar{x}, \bar{y}),  g\rangle (\bar{x}),
\end{array}
\end{equation*}
\begin{equation*}
\widehat {\partial}^\infty \left(\varphi \circ g \right) (\bar{x})  = \{0\}. 
\end{equation*}
\end{theorem}

\begin{proof} This follows from Theorem~\ref{thm2} specified for \eqref{compo12232023} with observing that the calmness condition \eqref{calm} imposed on $G=g$ yields the semilocal near-isolated calmness in (a)  of Theorem~\ref{thm192024}, while ensuring the regular coderivative scalarization \eqref{scal} and thus the fulfillment of both  claimed formulas by \eqref{192024_2}.
\end{proof}

As consequences of Theorem~~\ref{Theo12232023}, we get the following {\em product and quotient rules} under the calmness assumptions.

\begin{corollary}[\bf products and quotient rules]\label{product} Let the functions $\varphi_1,\varphi_2: X \to  \mathbb{\bar R}$ be calm at $\bar{ x}$. The following assertions hold:
\begin{description}
\item[\bf(i)] $\widehat {\partial} \left(\varphi_1 \cdot \varphi_2 \right) 
(\bar{ x}) = \widehat {\partial} (\varphi_2 (\bar{x}) \varphi_1 + \varphi_1 
(\bar{x}) \varphi_2) (\bar{ x})$ and $\widehat {\partial}^\infty \left( 
\varphi_1 \cdot \varphi_2 \right) (\bar{ x}) =\{0\}$.\\
If in addition $\varphi_1$ is Fr\'echet differentiable at $\bar{ x}$, then
\begin{align*}
\widehat {\partial} \left( \varphi_1 \cdot \varphi_2 \right) (\bar{x}) = \nabla 
\varphi_1 (\bar{x}) \varphi_2 (\bar{x}) + \widehat {\partial} \left(\varphi_1
(\bar{x}) \varphi_2 \right) (\bar{x}).
\end{align*}

\item[\bf(ii)] Assume that $\varphi_2(\bar{ x})\ne 0$. Then
\begin{align*}
\widehat {\partial} \left(\varphi_1/\varphi_2\right) (\bar{x}) = \dfrac{
\widehat {\partial} \left(\varphi_2(\bar{ x}) \varphi_1 - \varphi_1 (\bar{x})
\varphi_2\right) (\bar{ x})}{[\varphi_2(\bar{ x})]^2}, \quad 
\widehat{\partial}^\infty \left(\varphi_1  / \varphi_2 \right) (\bar{ x}) =\{0\}.
\end{align*}
If in addition $\varphi_1$ is Fr\'echet differentiable at $\bar{x}$, then
\begin{align*}
\widehat {\partial} \left(\varphi_1/\varphi_2\right) (\bar{ x}) = \dfrac{\nabla 
\varphi_1 (\bar{x}) \varphi_2(\bar{ x}) + \widehat {\partial} \left(- \varphi_1 
(\bar{ x}) \varphi_2 \right)(\bar{ x})}{[\varphi_2(\bar{ x})]^2}. 
\end{align*}

\item[\bf(iii)] Let $\varphi: X \to \mathbb{\bar R}$ be calm at $\bar{x}$ 
 with $\varphi(\bar x) \neq 0$. Then
\begin{align*}
\widehat {\partial} (1/\varphi) (\bar{x}) = - \dfrac{\widehat {\partial}^{+} 
\varphi (\bar{x})}{[\varphi(\bar{ x})]^2}, \quad \widehat {\partial}^\infty 
\left(1 /\varphi_2 \right) (\bar{ x}) = \{0\}.
\end{align*}
\end{description}
\end{corollary}

\begin{proof}
To verify (i), observe that $\varphi_1\cdot \varphi_2$ is represented as composition 
\eqref{compo12232023} with $\varphi: \mathbb{ R}^2 \to\mathbb{ R}$ and $g: X \to \mathbb{R}^2$ defined by $\varphi(y_1,y_2):=y_1\cdot y_2$ and $g(x) := (\varphi_1(x), \varphi_2(x))$. Thus 
Theorem~\ref{Theo12232023} yields the first equalities in (i), which implies 
the second one due to the elementary sum rule from \cite[Proposition~1.107(i)]{Mor2006}. 

The proof of (ii) is similar with $\varphi(y_1,y_2):=y_1/ y_2$ and the same  mapping $g$ in (\ref{compo12232023}), while (iii) is a special 
case of (ii) with $\varphi_1:= 1$ and $\varphi_2: = \varphi$.
\end{proof}

The next results provide new {\em sum rules} for both regular and singular regular subdifferentials of nonsmooth functions under the simultaneous fulfillment of the calmness and metric qualification conditions. 

\begin{theorem}[\bf sum rules for nonsmooth functions]\label{sum-rule}
 Let $\varphi_1,\varphi_2\colon X \to \mathbb{\bar R}$ be 
 extended-real-valued functions such that $\varphi_2$ is calm at $\bar{x}$ and that $\varphi(x,y) := \varphi_1(x)+y$ for all $(x,y)\in X\times \mathbb{R}$ is finite at $(\bar{x}, \varphi_2 (\bar{x}))$. Suppose in addition that the sets 
 ${\rm epi}\,\varphi$ and $\gr \varphi_2\times \mathbb{R}$ satisfy the metric qualification  condition \eqref{Q2} for $\partial^{\textasteriskcentered} = \widehat {\partial}$ at $(\bar{x}, \varphi_2 
 {(\bar x)}, \varphi_1(\bar{ x}) +\varphi_2(\bar x))$. Then we have
\begin{align*}
&\widehat {\partial} (\varphi_1 + \varphi_2) (\bar{x}) = \widehat {\partial} 
\varphi_1 (\bar{x}) + \widehat {\partial} \varphi_2 (\bar{x}),\\
& \widehat {\partial}^{\infty} (\varphi_1 + \varphi_2) (\bar{ x})  = \widehat {\partial}^{\infty} \varphi_1 (\bar{ x}).
\end{align*}
\end{theorem}

\begin{proof}
 It follows from \eqref{compo12232023} that
 $$
 \mu(x) = \left(\varphi \circ \varphi_2 \right) (x) = \varphi(x, \varphi_2(x)) = \varphi_1 (x) + \varphi_2(x),
 $$ 
 and thus the result follows by applying Theorem \ref{thm5} with $G:=\varphi_2$.
\end{proof}

\begin{theorem}[\bf chain rules for subdifferentials]\label{theo172024}
 Let $h: Y \to \mathbb{ \bar R}$ be finite at $ \bar y$ with $\bar{y} := 
 g(\bar x)$ and $g: X \to Y$ be  calm at $\bar{x}$. Suppose that the 
 sets $X \times \ep h$ and $\gr g \times \mathbb{R}$ satisfies 
 condition \eqref{Q2} for $\widehat{\partial}$ at $(\bar{x}, \bar{y}, h(\bar{y}))$. Then
 \begin{align}
& {\widehat {\partial}} \left(h \circ g \right) (\bar{x})  = 
\bigcup_{y^{\ast} \in {\widehat {\partial}} h(\bar{y})} \widehat {\partial}
\langle y^{\ast}, g \rangle (\bar{x}), \label{eq:cr09} \\
& {\widehat {\partial}^\infty} \left(h \circ g \right) (\bar{x}) = 
\bigcup_{y^{\ast} \in {\widehat {\partial}^\infty} h(\bar{y})} 
\widehat {\partial} \langle y^{\ast}, g\rangle(\bar{x}). \label{eq:cr10}
 \end{align} 
\end{theorem}

\begin{proof}
 For the marginal function $\mu(\cdot)$ in (\ref{1722023_1}), we consider 
 $\varphi: X \times Y\to \mathbb{ R}\cup\{\infty\}$ defined by 
 $\varphi(x,y) := h(y)$ for every $y\in Y$, and $G:=g$. Then both chain rules \eqref{eq:cr09} and \eqref{eq:cr10} follow from the corresponding results of
 Theorem~\ref{thm5}.  
\end{proof}

To get an effective consequence of Theorem~\ref{theo172024} in the case of real-valued functions $g$, the next lemma of its own interest is useful. The first equivalence below can be deduced from \cite[Theorem~1.86]{Mor2006} while we give another, {\em  variational proof}. The second equivalence seems never been mentioned in the literature.

\begin{lemma}[\bf regular normals to epigraphs and graphs]\label{lem172023}
 Let $f : X \to \mathbb{ \bar R}$ be finite at $\bar{ x}$, let $(\bar x, \alpha) 
\in \ep f$, and let $0\ne\lambda\in\mathbb R$. Then the following hold:
 \begin{equation}\label{eq:cr11}
(x^*, - \lambda) \in \widehat N \big((\bar{ x}, \alpha); \ep f\big) 
\Longleftrightarrow \lambda > 0, \, \alpha = f(\bar{x}), \, x^* \in 
\widehat {\partial} (\lambda f) (\bar{ x}),
\end{equation}
\begin{equation}\label{eq:cr12}
 (x^*, - \lambda) \in \widehat N\big((\bar{x}, f(\bar{x})); \gr f) 
\Longleftrightarrow   x^* \in \widehat {\partial} \left( 
\lambda f \right) (\bar{x}\big).
 \end{equation}
\end{lemma}

\begin{proof} 
 Observe that implication ``$\Longleftarrow$" is straightforward in  both cases of \eqref{eq:cr11} and \eqref{eq:cr12}.  To verify ``$\Longrightarrow$" in \eqref{eq:cr11}, take $(x^*, - \lambda) 
 \in \widehat N((\bar{ x},\alpha);\ep f)$ and find by Lemma~\ref{12202023_lem} a Fr\'echet differentiable function $s: X \times \mathbb{ R} 
 \rightarrow \mathbb{R}$ with $\nabla s(\bar{x}, \alpha) = (x^{\ast}, 
 -\lambda)$ and $s(x, r) \le s(\bar{x}, \alpha) = 0$ for all $(x, r) \in \ep f$. It follows from the Fr\'echet differentiability of $s(\cdot)$ that for every $\varepsilon > 0$ there exists $0 < \delta_1$ such that 
 \begin{equation*}
|s(x,r) - \langle(x^{\ast}, -\lambda), (x -\bar{x}, r-\alpha)\rangle\lvert \le 
\varepsilon(\|x -\bar{x}\|+|r-\alpha|)
 \end{equation*}
 for every $(x, r) \in \mathbb{B} 
 ((\bar{x}, \alpha), \delta_1)$. This yields 
\begin{equation}\label{12262023}
\langle(x^{\ast}, -\lambda), (x -\bar{x}, r-\alpha)\rangle\le \varepsilon 
(\|x -\bar{x}\|+|r-\alpha|)
\end{equation}
whenever  $(x, r) \in \mathbb{B} ((\bar{x}, \alpha), 
\delta_1) \cap \ep f$. Note that $(\bar{x}, \beta+\alpha)\in \ep f$ for any $\beta\ge 0$, and so for all small 
$\beta$ we have $(\bar{x}, \beta + \alpha) \in 
\mathbb{B} ((\bar{x}, \alpha), \delta_1) \cap \ep f$ and $\langle \beta, - 
\lambda \rangle \le \varepsilon \beta$. Thus $- \lambda \le \varepsilon$, and hence  $\lambda>0$ since $\varepsilon > 0$ was chosen arbitrarily. 

Next we show that $\alpha = f(\bar{x})$. Indeed, suppose on the contrary
that $\alpha > f(\bar{x})$. Then for any $\beta > 0$ satisfying $\beta < 
\delta_1$ and $\alpha - \beta > f(\bar{x})$, we have $(\bar{x}, \alpha - 
\beta) \in \mathbb{B} ((\bar{x}, \alpha), \delta_1) \cap \ep f$, and thus it 
follows from (\ref{12262023}) that $\lambda \beta \leq \varepsilon \beta$.
This brings us to the contradiction $\lambda \le 0$ by the arbitrary choice of $\varepsilon>0$.

To verify now implication ``$\Longrightarrow$'' in \eqref{eq:cr12}, pick $(x^*, - \lambda) \in\widehat N ( (\bar x,f ( \bar x));\gr f)$ with $\lambda\ne 0$. It follows from the definition that for each $\varepsilon>0$ there exists $\delta>0$ such that, whenever $x \in \mathbb B(\bar{x}, \delta)$, we have
\begin{align}\label{eq:55}
\langle x^*, x - \bar{x} \rangle + \langle - \lambda, f(x) - f(\bar{x}) \rangle 
\le \varepsilon \left(\|x - \bar x\| + |f(x) - f(\bar{x})| \right).
\end{align}
 To proceed further, consider the following two cases.\vspace*{0.05in}
 
 \noindent {\em Case~{\rm 1}, where $\lambda>0$}. Suppose without loss of generality that $\varepsilon<\lambda$. For any $x \in \mathbb{B} 
 (\bar{x}, \delta)$ with $(x, \alpha) \in \ep f$, we 
 deduce from \eqref{eq:55}  that
 \begin{align*}
 \langle x^*, x - \bar{x} \rangle - \langle \lambda, \alpha - f(\bar{x}) \rangle 
  & \le - \langle \lambda, \alpha - f(x) \rangle + \varepsilon \left( \|x - \bar x\| 
  + |f(x) - f(\bar{x})| \right) \\
  & \le \varepsilon \left(\|x - \bar x\| + |\alpha - f(\bar{x})| \right).
 \end{align*}
 This yields $(x^*, - \lambda) \in \widehat N ((\bar{x}, f(\bar{x})); \ep f)$, and thus it follows from \eqref{eq:cr11} that 
 $x^* \in \widehat {\partial} \left(\lambda f\right) (\bar{x})$, which justifies \eqref{eq:cr12} in this case. \vspace*{0.05in}
 
\noindent {\it Case~{\rm 2}, where $\lambda<0$.} We may assume that $\varepsilon<-\lambda$. For any $x \in \mathbb{B} (\bar{x}, \delta)$ with $(x, \alpha) \in \ep (-f)$, we derive from \eqref{eq:55} that
 \begin{align*}
\langle x^*, x - \bar{x} \rangle - \langle -\lambda, \alpha -(- f)(\bar{x}) \rangle 
& \le \langle \lambda, \alpha - (-f)(x) \rangle + \varepsilon \left( \|x - \bar x\| + |f(x) - f(\bar{x})| \right) \\
& \le \varepsilon \left(\|x - \bar x\| + |\alpha -(-f)(\bar{x})| \right).
 \end{align*}
 This yields $(x^*, \lambda) \in \widehat N ((\bar{x}, (-f)(\bar{x})); \ep (-f))$, which implies in turn that
 $x^* \in \widehat {\partial} \left(\lambda f\right) (\bar{x})$ due to \eqref{eq:cr11} and thus completes the proof of the lemma.  
\end{proof}

Finally, we are ready to derive, as a consequence of Theorem~\ref{theo172024}, the desired chain rules for both regular and singular regular subdifferentials of compositions with real-valued inner functions defined on normed spaces. 

\begin{corollary}[\bf chain rules with real-valued inner functions]\label{chain-cor} Let $h: \mathbb{R} \to \mathbb{\bar R}$ be finite at $ \bar y$ with 
$\bar y := g(\bar x)$, and let $g: X \to \mathbb{R}$ be calm at 
$\bar{x}$. Suppose that the sets $X\times \ep h$ and $\gr g \times 
\mathbb{R} $ satisfy condition \eqref{Q2} for $\widehat {\partial}$ 
at $(\bar{ x},\bar{ y},h(\bar{ y}))$. Then
\begin{align}
& ~~ {\widehat {\partial}} \left(h \circ g \right) (\bar{x}) = \left(\bigcup_{\lambda 
\in {\widehat {\partial}} h(\bar{y}),\,\lambda\ge 0} 
\lambda\widehat {\partial} g(\bar{x})\right)\cup\left(\bigcup_{\lambda 
\in {\widehat {\partial}} h(\bar{y}),\,\lambda< 0} 
\lambda\widehat {\partial}^+g(\bar{x})\right),\label{eq:cr13} \\
&{\widehat {\partial}^\infty} \left(h \circ g \right) (\bar{x}) = 
\left(\bigcup_{\lambda 
\in {\widehat {\partial}^\infty } h(\bar{y}),\,\lambda\ge 0} 
\lambda\widehat {\partial}g(\bar{x})\right)\cup\left(\bigcup_{\lambda 
\in {\widehat {\partial}^\infty } h(\bar{y}),\,\lambda< 0} 
\lambda\widehat {\partial}^+g(\bar{x})\right).\label{eq:cr15}
 \end{align}
\end{corollary}

\begin{proof} Observe that, by the scalarization formula \eqref{scal} under the imposed calmness assumption, the subgradient set standing on the right-hand sides of \eqref{eq:cr09} and \eqref{eq:cr10} reduces to $D^*g(\bar x)(y^*)$ with $y^*:=\lambda\in\R$ due to $Y=\mathbb{R}$ in the setting of the corollary. The coderivative definition \eqref{cod} tells us that $x^{*} \in \widehat {D}^{*} g(\bar{x}) (\lambda)$ if and only if $(x^{*}, - \lambda) \in \widehat N((\bar{x}, \bar{y}); \gr g)$. Consider the following two cases.\vspace*{0.05in} 
 
\noindent {\em Case~{\rm 1}, where $\lambda = 0$}. Since $g$ is calm at 
$\bar{x}$, we have $x^{*} = 0$ by scalarization \eqref{scal}.  Then the sets on the right-hand sides of \eqref{eq:cr13} and \eqref{eq:cr15} are obviously zeros, while the left-hand sides reduce to zeros by \eqref{eq:cr09} and \eqref{eq:cr10} in Theorem~\ref{theo172024}.\vspace*{0.05in}

\noindent {\em Case~{\rm 2}, where $\lambda\ne 0$}. In this case, it follows from \eqref{eq:cr12} in Lemma~\ref{lem172023} that $x^{*} \in\widehat {\partial} (\lambda g) (\bar{x})$. Using the latter inclusion and applying Theorem~\ref{theo172024} with taking into account definition \eqref{upper-sub} of the upper regular subdifferential, we arrive at both formulas \eqref{eq:cr13}, \eqref{eq:cr15} and thus complete the proof of the corollary.
\end{proof}\vspace*{-0.15in}

\section{Bilevel Programming in Asplund Spaces}\label{sec:bilevel}

This section is mainly devoted to deriving {\em necessary optimality conditions} in nonsmooth models of {\em optimistic bilevel programming}. Developing the {\em marginal functions approach} to bilevel programs and having in hands the results obtained above, we invoke here the following {\em two major ingredients} that are new in comparison with the previous sections. The {\em first} one is dealing with not general normed spaces while with the class of {\em Asplund spaces}, which is fairly broad (containing, in particular, every reflexive Banach spaces) and has been well recognized and applied in variational analysis. The {\em second} major ingredient of our developments is involving, along with regular normal/subdifferential constructions, the {\em limiting normals and subgradients} by Mordukhovich. 

Recall that a Banach space $X$ is {\em Asplund} if any separable subspace of $X$ has a separable dual. This class of spaces have been comprehensively investigated in geometric theory of Banach spaces admitting various characterizations and a great many applications to different aspects of variational analysis, optimization, and optimal control; see both volumes of \cite{Mor2006} and the references therein. Unless otherwise stated, all the spaces considered below are  assumed to be Asplund.\vspace*{0.05in}

Next we recall the limiting generalized differential constructions for general sets and functions introduced by the third author and largely used below; see, e.g.,  \cite{GO,Mor2006,Book_Penot,RW,Thibault} and their bibliographies.

\begin{definition}[\bf limiting normals and subgradients]\label{lim}
 {\rm Consider a  nonempty set $S \subset X$ and a function $f: X \rightarrow \bar{\mathbb{R}}$ with ${\rm dom}\,f<\infty$. Then:
\begin{itemize}
\item[\bf(i)] The {\it limiting normal cone} to $S$ at $\bar{x}\in S$ is defined by
\begin{align}\label{lnc}
N (\bar{x}; S) := \Lsup\limits_{x \rightarrow \bar{x}} \widehat N ({x}; S).
\end{align}

\item[\bf(ii)] The {\it limiting subdifferential} of $f$ at 
$\bar x\in{\rm dom}\,f$ is defined by
\begin{align}\label{lim-sub}
\partial f(\bar x) :=\big\{x^{*} \in X^{*}\;\big|\;(x^{*}, -1) \in N \big((\bar{x}, 
f(\bar x)\big);\ep f)\big\}.
\end{align}

\item[\bf(iii)] The {\it singular limiting subdifferential} of $f$ at $\bar{x}\in\dom f$ is defined by
\begin{align}\label{sin-lim}
\partial^\infty f(\bar x) :=\big\{x^{*} \in X^{*}\;\big|\; (x^{*}, 0) \in N 
\big((\bar{x}, f(\bar x)); \ep f\big)\big\}.
\end{align} 
\end{itemize}}
\end{definition}

Definition~\ref{lim}  presents geometric constructions of limiting and singular limiting  subgradients. There are several equivalent representations of the subdifferentials \eqref{lim-sub} and \eqref{sin-lim} that can be founded in the aforementioned books but are not employed in what follows. It is important to emphasize that the sets in \eqref{lnc}--\eqref{sin-lim} are often {\em nonconvex} while enjoying comprehensive {\em calculus rules} based on {\em variational/extremal principles} of variational analysis.\vspace*{0.05in}

Now we proceed with formulating the bilevel programming model of our study. The {\em lower-level problem} of parametric optimization is the following:
\begin{align}\label{POP_ll}
 \min \varphi (x, y) \quad \text{subject to $y \in G(x)$},
\end{align}
where the lower-level cost function $\varphi: X \times  Y \to\mathbb{ \bar R}$ depend on the decision variable $y\in Y$ and parameter $x\in X$, and where $G: X \rightrightarrows Y$ is the parameter-dependent constraint mapping. The {\em lower-level optimal value function} for \eqref{POP_ll} is the marginal function $\mu(\cdot)$ defined in \eqref{1722023_1}, and the corresponding {\em lower-level  solution mapping} $M(\cdot)$ is taken from \eqref{1722023_2}.

The {\em upper-level problem} is formulated as follows:
\begin{align*}
\min \psi (x, y) \quad \text{subject to $y \in M(x)$ for each 
$x \in \Omega \subset X$},
\end{align*}
where $\psi: X \times Y \to \mathbb{ R}$ is the upper-level cost function, and where the underlying, most involved constraint  $y \in M(x)$ is formed by the solution mapping $M(\cdot)$ associated with the lower-level problem \eqref{POP_ll}.

Using the marginal function approach, we finally define the hierarchical 
optimization problem known as the {\em optimistic bilevel model} (see, e.g.,
\cite{DMZ2012,DZ,Dempe2020,Ye2,Ye} with the discussions and references therein):
\begin{align}\label{OP}
\min\vartheta(x)\quad\text{subject to $x\in\Omega$, where  
$\vartheta(x):=\inf\left\{\psi(x,y)\,\,|\,\,y\in M(x)\right\}$}.
\end{align}
For simplicity, we consider below the case of \eqref{OP} with $\Omega=X$ and
\begin{align}\label{G}
G(x) :=\big\{y \in Y\;\big|\;g_i(x,y)\le 0\;\mbox{ for all }\;i=1,\ldots,p 
\big\}.
\end{align}
Having this in mind, we can rewrite \eqref{OP} in the following form:
\begin{align}\label{SLMP}
\min\psi(x,y)\quad\text{subject to $g_i(x,y)\le 0$, $i=1,\ldots,p$, and $\varphi(x,y)\le \mu(x)$,}
\end{align}
where the marginal function $\mu(\cdot)$ is intrinsically nonsmooth and generally nonconvex even for nice initial data $\varphi$, $\psi$, and $g_i$. It has been realized that standard constraint qualifications fail to fulfill in optimal solutions of \eqref{SLMP}. In this setting, a new constraint qualification  was introduced in \cite{Ye} under the name of ``partial calmness". To define it, consider the {\em perturbed} version of \eqref{SLMP} given by \begin{equation}\label{SLMP_1}
\begin{array}{ll}
\min \psi (x, y) \quad\text{subject to $g_i (x, y) \le 0$, $i=1,\ldots,p$,}\\ 
\text{and $\varphi (x, y) - \mu(x) + \nu = 0$ as $\nu \in \mathbb{ R}$.}
\end{array}
\end{equation}

\begin{definition}[\bf partial calmness]\label{par-calm} {\rm 
 Problem~\eqref{SLMP} is {\em partially calm} at its local optimal solution 
 $(\bar x, \bar y)$ with modulus $\kappa>0$ if there exists a neighborhood 
 $U$ of $(\bar x, \bar y, 0) \in X \times Y \times \mathbb{ R}$ such that
 \begin{align*}
\psi (x, y) - \psi (\bar x, \bar y) + \kappa |\nu| \ge 0
\end{align*}
 for all the triples $(x,y,\nu)\in U$ feasible to \eqref{SLMP_1}.}
\end{definition} 

We refer the reader to \cite{BM1,BM2,DMZ2012,DZ,HS,MMZ,M2018,Dempe2020,Ye} for various explicit conditions ensuring partial calmness. The essence of partial calmness is the possibility of {\em exact penalization} of the troublesome constraint $\varphi(x,y)\le\mu(x)$ in \eqref{SLMP}, which was observed in \cite{Ye} in finite-dimensional spaces. The following statement is taken from \cite[Proposition~7.8.12]{Dempe2020}, where the proof holds in general normed spaces.

\begin{lemma}[\bf exact penalization]\label{PP}
Let \eqref{SLMP} be partially calm at its local optimal 
solution $(\bar x, \bar y)$ with modulus $\kappa>0$, and let $\psi$ be
continuous at $(\bar x, \bar y)$. Then $(\bar x, \bar y)$ is a local optimal 
solution to the penalized problem
\begin{equation*}\label{SLP_P}
\begin{array}{ll}
& \min \psi(x, y) + \kappa\big(\varphi(x, y) - \mu(x)\big)\\ 
& {\rm subject\; to }\;g_i(x,y)\le 0,\;i=1,\ldots,p.
\end{array}
\end{equation*}
\end{lemma}

The next lemma, valid in normed spaces, is taken from \cite[Proposition~4.1]{MNY2006}.

\begin{lemma}[\bf optimality conditions for difference functions]\label{DR}
 Let $\varphi_1, \varphi_2: X \to \mathbb{\bar R}$ be functions
 finite at $\bar{x}$, and let $\widehat{\partial} \varphi_2(\bar{ x}) \ne
\emptyset$. Then any local minimizer $\bar x$ of the difference function
 $\varphi_1-\varphi_2$ satisfies the regular subdifferential inclusion
\begin{align}\label{NC_DR}
\widehat{\partial} \varphi_1 (\bar x) \subset \widehat{\partial}
\varphi_2 (\bar{ x}).
\end{align}
\end{lemma}

The last lemma here, taken from  \cite[Theorems~3.8 and 
3.86]{Mor2006}, gives us a useful representation of limiting normals to sets defined by inequality constraints.

\begin{lemma}[\bf limiting normals to inequality constraints]\label{Lem:Lipsch}
Let $g_i\colon X\times Y\to\mathbb{R}$ be locally Lipschitzian around $({\bar{ x}}, \bar{y})$, and let $G_i:=\{(x,y)\in X\times Y\,\,|\,\, g_i(x,y)\le 0 \}$ for all $i=1,\ldots, p$. Denote $I (\bar x, \bar{y}): = \{i \in \{1, \ldots, p\} \,\,|\,\, g_i (\bar x, \bar{y}) = 0\}$ and impose the constraint 
qualification condition
\begin{equation}\label{BCQ}
\begin{array}{ll}
\displaystyle\Big[\sum_{i \in I (\bar x, \bar{ y})} \lambda_i (x^*_i, y^*_i) = 0, \, (x^*_i, y^*_i) \in 
\partial g_i (\bar{ x}, \bar{ y}),\, \lambda_i \ge 0,\, i \in I (\bar{ x},\bar{ y})\Big]\\ 
\Longrightarrow\big[\lambda_i = 0\;\mbox{ for all }\;i \in I (\bar{ x}, \bar{ y})\big].
\end{array}
\end{equation}
Then for any $(x^*, y^*) \in  N \left((\bar x,\bar y);\bigcap_{i=1}^{p} G_i \right)$
we have the inclusion
\begin{align*}
(x^*, y^*)  \in \sum_{i \in I (\bar x,\bar{ y})} \lambda_i \partial g_i(\bar x,\bar{ y}) 
\end{align*}
with $\lambda_i \ge 0$ whenever $i\in I(\bar x,\bar y)$.
\end{lemma}

Now we are ready to derive necessary optimality conditions for optimistic bilevel programs in arbitrary Asplund spaces. 

\begin{theorem}[\bf necessary optimality conditions for bilevel programs]
\label{NOC:Asplund} $\,$ Let $(\bar x,\bar y)$ be a local optimal solution to the bilevel program \eqref{SLMP}, where $\varphi$ is lower semicontinuous while $g_i$ for $i=1,\ldots,p$ are locally Lipschitzian around $(\bar x, \bar y)$, where $\psi$ is Fr\'echet differentiable at this point, and where \eqref{SLMP} is partially calm at $(\bar x, \bar y)$ with modulus $\kappa>0$. Suppose in addition that the constraint qualification \eqref{BCQ} holds, that
$\widehat{\partial} \mu(\bar x) \ne \emptyset$, and that
\begin{itemize}
\item[\bf$(\cal{A}$)] the sets ${\rm epi}\,\varphi$ and 
$\gr G \times \mathbb{R}$ for $G$ from \eqref{G} satisfy the metric qualification condition \eqref{Q2} for $ \partial^{\textasteriskcentered}=\partial$  at $(\bar{x}, \bar{y}, \varphi(\bar{x}, \bar{y}))$.
\end{itemize}
Then for each $u^* \in \widehat{\partial}\mu(\bar x)$, there exist 
multipliers $\lambda_i$ and $\beta_i$ with $i=1,\ldots,p$ such that we have the following necessary optimality conditions:
\begin{align}\label{NOC:1}
(u^*,0) \in \partial \varphi(\bar{x}, \bar{y}) + \sum_{i=1}^{p} \beta_i 
\partial g_i (\bar{ x}, \bar y), 
\end{align}
\begin{align}\label{NOC:2}
(u^*, 0) \in \kappa^{-1} \nabla \psi (\bar{x}, \bar{y}) + \partial \varphi 
(\bar{x}, \bar{y}) + \sum_{i=1}^{p} \lambda_i \partial g_i (\bar{x}, 
\bar y), 
\end{align}
\begin{align}\label{CSC}
\lambda_i,\beta_i \ge 0, \quad \lambda_i g_i (\bar{x}, 
\bar{y}) = \beta_i g_i (\bar{x}, \bar{y}) = 0\;\mbox{ as }\;i=1,\ldots, p.
\end{align}
\end{theorem}

\begin{proof}
From Lemma~\ref{PP} and the infinite constraint penalization via the
indicator function $\delta(\cdot; \gr G)$, we get that the pair $(\bar x, \bar y)$ be a local optimal solution to the single-objective unconstrained problem
\begin{align}\label{3192024}
\min\big[\psi (x, y) + \kappa\big( \varphi (x, y) - \mu(x)\big) + \delta
(\cdot, \gr G)\big].
\end{align}
Using the necessary optimality condition \eqref{NC_DR} from 
Lemma~\ref{DR} for the difference function in \eqref{3192024} under $\widehat\partial\mu(\bar x)\ne\emptyset$ together with the elementary sum rule for regular subgradients in
\cite[Proposition~1.107(i)]{Mor2006}, we have
\begin{equation}\label{1932024_eq:1}
\begin{array}{ll}
\big( \kappa \widehat{\partial} \mu(\bar{x}), 0 \big) & \subset
\nabla \psi (\bar x, \bar y) + \widehat{\partial}\big( \kappa
\varphi (\cdot) + \delta(\cdot,\gr G)\big) (\bar x, \bar y)\\
& \subset \nabla \psi (\bar x, \bar y) + \kappa {\partial} \big(
\varphi (\cdot) + \delta (\cdot, \gr G)\big) (\bar x, \bar y).
\end{array}
\end{equation}
Moreover, recalling that $\mu(x)=\inf_{y\in Y}(\varphi(\cdot)+\delta(\cdot,\gr G))(x,y)$ and then applying \cite[Theorem 4.47]{Book_Penot} tell us that
\begin{equation}\label{1932024_eq:2}
\begin{array}{ll}
\big(\widehat{\partial} \mu(\bar{x}), 0\big) & \subset \widehat{\partial}
\big(\varphi(\cdot) + \delta(\cdot, \gr G)\big) (\bar x, \bar y)\\
&\subset{\partial}\big( \varphi(\cdot) + \delta(\cdot, \gr G)\big)
(\bar x, \bar y).
\end{array}
\end{equation}
Observing that $\partial\delta(s;S)=N(s;S)$ for any set $S$ with $s\in S$  and using
\begin{align*}
\ep (\varphi + \delta_{\gr G}) = \ep \varphi \cap [\gr G \times \mathbb{ R}]\quad\text{and}\quad  (\varphi + \delta_{\gr G})(\bar x, \bar y)=\varphi(\bar x,\bar y),
\end{align*}
we therefore arrive at the equality
\begin{align*}
N \big( (\bar{x}, \bar{y}, (\varphi + \delta_{\gr G})(\bar x, \bar y)) ;
\ep \varphi\cap [\gr G\times \mathbb{R}] \big)\\= N \big( (\bar{x}, \bar{y}, \varphi (\bar{x}, \bar{y}));
\ep \varphi\cap [\gr G\times \mathbb{R}]\big).
\end{align*}
Since assumption $(\cal A)$ is satisfied, the result of \cite[Theorem~3.1]{HuxuAn} ensures that
\begin{align*}
&N \big( (\bar{x}, \bar{y}, \varphi (\bar{x}, \bar{y}));
\ep \varphi\cap [\gr G\times \mathbb{R}] \big)\\&=N \big( (\bar{x}, \bar{y}, \varphi (\bar{x}, \bar{y})\big);
\ep \varphi)
+N \big( (\bar{x}, \bar{y}, \varphi (\bar{x}, \bar{y}));
\gr G\times \mathbb{R}\big).
\end{align*}
Combining the latter with \eqref{1932024_eq:1} and \eqref{1932024_eq:2} brings us to
\begin{align*}
\begin{cases}
\big( \kappa \widehat{\partial} \mu(\bar{x}), 0 \big) & \subset \nabla
\psi (\bar x, \bar y) + \kappa{\partial} \varphi (\bar x, \bar y) + N \big(
(\bar x, \bar y); \gr G \big), \\
\big( \widehat{\partial} \mu (\bar{x}), 0 \big) & \subset {\partial}
\varphi (\bar x, \bar y) + N \big((\bar x, \bar y); \gr G \big).
\end{cases}
\end{align*}
Employing finally Lemma~\ref{Lem:Lipsch} under the assumed constraint qualification \eqref{BCQ}, we readily get the existence of $u^* \in \widehat{\partial}\mu(\bar{ x})$ together with multipliers $\lambda_i$ and $\beta_i$ satisfying all the conditions in \eqref{NOC:1}--\eqref{CSC} and thus complete the proof.
\end{proof}

Now we present two remarks comparing Theorem~\ref{NOC:Asplund} with previous results in this direction and discussing the major assumptions imposed.

\begin{remark}[\bf discussions on optimality conditions]\label{compar} 
{\rm Several types of necessary optimality conditions for optimistic bilevel programs have been obtained in the literature; see, e.g., \cite{DMZ2012,DZ,M2018,Dempe2020,Ye} with further references and discussions. The optimality conditions derived therein are all in finite-dimensional spaces, while Theorem~\ref{NOC:Asplund} is established in general Asplund space setting. The closest statement to Theorem~\ref{NOC:Asplund} is given in \cite[Theorem~7.9.3]{Dempe2020}, where the same necessary optimality conditions \eqref{NOC:1}--\eqref{CSC} are also obtained while under essentially weaker assumptions even in finite dimensions. Indeed, we do not require here, in contrast to \cite{Dempe2020}, that the lower-level cost $\varphi$ is locally Lipschitzian and that the solution mapping $M$ is inner semicontinuous. Moreover, our assumption (${\cal A}$) significantly improves the qualification conditions in \cite[Theorem~7.9.3]{Dempe2020}. Observe finally that Theorem~\ref{NOC:Asplund} is free of SNC/PSNC assumptions (see \cite{Mor2006}), which are usually imposed in deriving necessary optimality conditions for non-Lipschitzian infinite-dimensional problems.}
\end{remark}

The next remark discusses the two major assumptions of Theorem~\ref{NOC:Asplund}.

\begin{remark}[\bf discussing major assumptions]\label{assump}  $\,${\rm 

\begin{itemize}
\item[\bf(i)] The metric qualification condition \eqref{Q2} for limiting subgradients has been well recognized and largely employed in variational analysis and optimization; in particular, it significantly improves other qualification conditions used in subdifferential calculus and optimization theory; see, 
e.g., \cite{HuxuAn,Ioffe,NgaiThera,Book_Penot,Thibault} and compare with \cite{Mor2006}.

\item[\bf(ii)] Regarding the condition $\widehat{\partial}\mu(\bar x)\ne\emptyset$, observe first that this assumption holds when $\mu$ is convex, but it is way too much to assume. In particular, it follows from \cite[Propositions~4.4 and 4.6]{Book_Penot} that if $X$ is finite-dimensional and $\mu(\cdot)$ is {\em tangentially convex} as in \cite{Book_Penot}, then we have $\widehat{\partial}\mu(\bar x)\ne\emptyset$ {\em if and only if} the function $\mu(\cdot)$ is {\em calm}  at $\bar x$ in the sense of \eqref{calm}. Furthermore, the {\em lower regularity} of $\mu(\cdot)$ at $\bar x$, in the sense that $\widehat{\partial}\mu(\bar x)=\partial\mu(\bar x)$, ensures the nonemptiness of $\widehat{\partial}\mu(\bar x)$ in Asplund spaces under a certain mild SNC property; see \cite[Corollar~2.24]{Mor2006}. The latter class of functions is sufficiently broad including, e.g., maximum, semiconvex, lower-${\cal C}^1$ functions, etc.; see the book \cite{Thibault} with many references therein.
\end{itemize}}
\end{remark}

We conclude this section with the following simplification of Theorem~\ref{NOC:Asplund} in the case of bilevel programs with differentiable data. Recall that a function $f\colon Z\to\mathbb\R$ defined on a normed space $Z$ is {\em strictly differentiable} at $\bar z$ with the strict derivative $\nabla f(\bar z)$ if 
\begin{equation*}
\lim_{z,u\to\bar z}\frac{f(z)-f(u)-\langle\nabla f(\bar z),z-u\rangle}{\|z-u\|}=0.
\end{equation*}
It is well known that this notion lies between Fr\'echet differentiable of $f$ at $\bar z$ and continuous differentiability of $f$ around this point. If all the constraint functions $g_i,\;i=1,\ldots,p$, in \eqref{SLMP} are strictly differentiable at $(\bar x,\bar y)$, that then the \textit{qualification condition} \eqref{BCQ} reduces to the classical {\em Mangasarian-Fromovitz constraint qualification} (MFCQ). 

\begin{corollary}[\bf bilevel programs with differentiable data]\label{cor:smooth}
 Let $(\bar x, \bar y)$ be a local optimal solution to the bilevel program
 \eqref{SLMP} under the fulfillment of $({\cal A})$. Assume that the functions $g_i$, $i=1,\ldots,p$, and $\varphi$ are strictly differentiable at $(\bar x,\bar y)$, that $\psi$ is Fr\'echet differentiable at $(\bar x, \bar y)$, that \eqref{SLMP} is partially calm at $(\bar x, \bar y)$ with modulus $\kappa>0$, that MFCQ holds at $(\bar x,\bar y)$, and that $\widehat{\partial} \mu (\bar x)\ne\emptyset$. Then there exist  multipliers $\lambda_i$ and $\beta_i$ for $i=1, \ldots,p$ such that
\begin{equation*}
\begin{array}{ll}
\displaystyle\nabla_x \psi (\bar{x}, \bar{y}) + \sum_{i=1}^{p} \kappa \left( \lambda_i - \beta_i \right) \nabla_x g_i (\bar{x}, \bar y) = 0,\\
\nabla_y\psi(\bar{ x},\bar{ y})+
\displaystyle\sum_{i=1}^{p} \kappa \left( \lambda_i 
- \beta_i \right) \nabla_y g_i(\bar{ x},\bar y) = 0,\\
\kappa^{-1} \nabla_y \psi (\bar{ x}, \bar{ y}) + \nabla_y\varphi (\bar{ x},\bar{ y}) +
\displaystyle\sum_{i=1}^{p} \lambda_i \nabla_y g_i (\bar{x}, 
\bar y) = 0,\\
\nabla_y\varphi(\bar{ x},\bar{ y})+
\displaystyle\sum_{i=1}^{p}\beta_i 
\nabla_y g_i(\bar{ x},\bar y)=0,\\
\lambda_i \ge 0,\;\beta_i \ge 0,\;\lambda_i g_i (\bar{x},
\bar{y}) = \beta_i g_i (\bar{x}, \bar{y}) = 0\;\mbox{ as }\;i = 1,\ldots, p.
\end{array}
\end{equation*}	
\end{corollary}
\begin{proof} Recall that the limiting subdifferential of a strictly differentiable function reduce to its strict derivative, and that any such function is locally Lipschitzian around the point in question. Thus the necessary optimality conditions formulated in the corollary follow from conditions \eqref{NOC:1}--\eqref{CSC} of Theorem~\ref{NOC:Asplund}.
\end{proof}

\section{Conclusions}\label{sec:7}

The paper provides a systematic study of regular subgradients and their singular counterparts for a general class of marginal functions defined on normed spaces with applications to variational calculus rules for such constructions. Further applications are given, together with using the corresponding limiting normals and subgradients, to deriving necessary optimality conditions for the optimistic model of bilevel programming in Asplund spaces via the marginal function approach. Although our developments are given in infinite-dimensional spaces, the obtained results are new in finite dimensions. In the future research, we will continue the lines of this study with applying the newly introduced nearly-isolated calmness properties to other aspects of variational analysis and optimization and aiming, in particular, at the pessimistic model of bilevel programming.

\bigskip

\noindent {\bf Acknowledgements:} 
The authors wishes to thank both reviewers for their comments, remark and criticism 
that helped to improve the quality of the paper.

\end{document}